\documentclass[12pt]{amsart}
\usepackage{verbatim,amssymb,latexsym,amscd}
\usepackage[all]{xy}

\usepackage[colorlinks,linkcolor=blue,citecolor=blue,urlcolor=red]{hyperref}

\newcommand{\ie}{{\it i.e. }}
\newcommand{\cf}{{\it cf. }}
\newcommand{\eg}{{\it e.g. }}
\newcommand{\loccit}{{\it loc. cit. }}
\newcommand{\resp}{{\it resp. }}

\newcommand{\A}{\mathbf{A}}
\renewcommand{\c}{\operatorname{C}}
\newcommand{\h}{\operatorname{h}}
\newcommand{\uR}{\underline{R}}
\newcommand{\uS}{\underline{S}}
\newcommand{\C}{\place/{\h}}
\newcommand{\D}{\Sm^{\proj}(F)/R}

\newcommand{\N}{\mathbf{N}}
\renewcommand{\P}{\mathbf{P}}
\newcommand{\Q}{\mathbf{Q}}
\newcommand{\R}{\mathbf{R}}

\newcommand{\Z}{\mathbf{Z}}
\newcommand{\sA}{\mathcal{A}}

\newcommand{\sC}{\mathcal{C}}
\newcommand{\sD}{\mathcal{D}}
\newcommand{\sE}{\mathcal{E}}
\newcommand{\sF}{\mathcal{F}}
\newcommand{\sG}{\mathcal{G}}
\newcommand{\sH}{\mathcal{H}}
\newcommand{\sI}{\mathcal{I}}

\newcommand{\sO}{\mathcal{O}}
\newcommand{\sP}{\mathcal{P}}
\newcommand{\sR}{\mathcal{R}}
\newcommand{\sS}{\mathcal{S}}

\newcommand{\bG}{\mathbb{G}}

\newcommand{\fm}{\mathfrak{m}}
\newcommand{\fp}{\mathfrak{p}}

\newcommand{\Spec}{\operatorname{Spec}}
\newcommand{\Ker}{\operatorname{Ker}}

\newcommand{\IM}{\operatorname{Im}}
\newcommand{\Bl}{\operatorname{Bl}}
\newcommand{\place}{\operatorname{\mathbf{place}}}
\newcommand{\dv}{\operatorname{\mathbf{dv}}}
\newcommand{\field}{\operatorname{\mathbf{field}}}

\newcommand{\pro}[1][]{\operatorname{pro}_{#1}\hbox{--}}

\newcommand{\Rat}{\operatorname{\mathbf{Rat}}}

\newcommand{\Supp}{\operatorname{Supp}}
\newcommand{\Map}{\operatorname{Map}}
\newcommand{\Hom}{\operatorname{Hom}}
\newcommand{\Sets}{\operatorname{\mathbf{Sets}}}

\newcommand{\trdeg}{\operatorname{trdeg}}

\newcommand{\Exc}{{\operatorname{Fund}}}
\newcommand{\Fund}{{\operatorname{Exc}}}

\newcommand{\Cat}{\operatorname{\mathbf{Cat}}}

\newcommand{\car}{\operatorname{char}}

\newcommand{\dom}{{\operatorname{dom}}}

\newcommand{\an}{{\operatorname{an}}}
\newcommand{\reg}{{\operatorname{reg}}}

\newcommand{\Funct}{\operatorname{\mathbf{Funct}}}

\newcommand{\VarP}{\operatorname{\mathbf{VarP}}}
\newcommand{\Norm}{\operatorname{\mathbf{Norm}}}

\newcommand{\SmP}{\operatorname{\mathbf{SmP}}}
\newcommand{\Sm}{{\operatorname{\mathbf{Sm}}}}
\newcommand{\Sch}{\operatorname{\mathbf{Sch}}}
\newcommand{\Set}{\operatorname{\mathbf{Set}}}

\newcommand{\proper}{{\operatorname{prop}}}
\newcommand{\qp}{{\operatorname{qp}}}

\newcommand{\proj}{{\operatorname{proj}}}

\newcommand{\Intsep}{\operatorname{\mathbf{Var}}}
\newcommand{\Var}{\operatorname{\mathbf{Var}}}

\newcommand{\op}{{\text{\rm op}}}

\newcommand{\ess}{{\text{\rm ess}}}

\newcommand{\ttto}{\dashrightarrow}
\newcommand{\tto}{\rightsquigarrow}
\newcommand{\To}{\longrightarrow}
\newcommand{\rr}{\operatornamewithlimits{\rightrightarrows}\limits}
\newcommand{\by}[1]{\overset{#1}{\longrightarrow}}
\newcommand{\yb}[1]{\overset{#1}{\longleftarrow}}
\newcommand{\iso}{\by{\sim}}
\newcommand{\osi}{\yb{\sim}}
\newcommand{\inj}{\hookrightarrow}
\newcommand{\Inj}{\lhook\joinrel\longrightarrow}
\newcommand{\surj}{\rightarrow\!\!\!\!\!\rightarrow}
\newcommand{\Surj}{\relbar\joinrel\surj} 
\newcommand{\colim}{\varinjlim}
\renewcommand{\lim}{\varprojlim}

\newcommand{\codim}{\operatorname{codim}}
\newcommand{\prf}{\noindent{\bf Proof. }}
\renewcommand{\qed}{\hfill $\Box$\medskip}

\renewcommand{\phi}{\varphi}
\renewcommand{\epsilon}{\varepsilon}

\newcounter{spec}
\newenvironment{thlist}{\begin{list}{\rm{(\roman{spec})}}%
{\usecounter{spec}\labelwidth=20pt\itemindent=0pt\labelsep=10pt}}%
{\end{list}}%

\newtheorem{Th}{Theorem}
\newtheorem{Conj}{Conjecture}

\swapnumbers

\newtheorem{thm}{Theorem}[subsection]
\newtheorem{lemma}[thm]{Lemma}
\newtheorem{prop}[thm]{Proposition}
\newtheorem{cor}[thm]{Corollary}

\newtheorem{thmr}[thm]{*Theorem}
\newtheorem{lemmar}[thm]{*Lemma}
\newtheorem{propr}[thm]{*Proposition}
\newtheorem{corr}[thm]{*Corollary}

\theoremstyle{definition}

\newtheorem{defn}[thm]{Definition}
\newtheorem{nota}[thm]{Notation}
\theoremstyle{remark}
\newtheorem{caution}[thm]{Caution}
\newtheorem{rk}[thm]{Remark}
\newtheorem{rks}[thm]{Remarks}
\newtheorem{qn}[thm]{Question}

\newtheorem{ex}[thm]{Example}
\newtheorem{exs}[thm]{Examples}
\newtheorem{exo}[thm]{Exercise}

\numberwithin{equation}{section}

\entrymodifiers={!!<0pt,0.7ex>+}

\begin{document}

\title{Birational geometry and localisation of categories}
\author{Bruno Kahn}
\address{IMJ-PRG\\ Case 247\\4 place Jussieu\\
75252 Paris Cedex 05\\France}
\email{bruno.kahn@imj-prg.fr}
\author{R. Sujatha}
\address{University of British Columbia\\Vancouver, BC
V6T1Z2\\Canada}
\email{sujatha@math.ubc.ca}
\date{October 2,  2014}
\thanks{The first author acknowledges the support of Agence Nationale de la Recherche (ANR) under reference ANR-12-BL01-0005 and the second author  that of an NSERC Grant. Both authors acknowledge the support of CEFIPRA project 2501-1.}

\keywords{Localisation, birational geometry, places, R-equivalence.}

\subjclass[2010]{14E05, 18F99}

\begin{abstract} We explore connections 
between places of function fields over a base field $F$ and birational morphisms between smooth $F$-varieties.  This is done by 
considering various categories of fractions involving function fields or 
varieties as objects, and constructing functors between these categories. The main result is that in the localised category 
$S_b^{-1}\Sm(F)$, where $\Sm(F)$ denotes the usual 
category of smooth varieties over $F$ and $S_b$ is the set of birational 
morphisms, the set of morphisms between two objects $X$ and $Y$, with $Y$ proper, is the 
set of $R$-equivalence classes $Y(F(X))/R$. 
\end{abstract}

\maketitle

\hfill With appendices by Jean-Louis Colliot-Th\'el\`ene and Ofer Gabber

\tableofcontents

\section*{Introduction}

Let $\Phi$ be a functor from the category of smooth proper varieties over a field $F$ to the category of sets. We say that $\Phi$ is \emph{birational}
if it transforms birational morphisms into isomorphisms. In characteristic $0$, examples of such functors
are obtained by choosing a function field $K/F$ and defining $\Phi_K(X)=X(K)/R$, the set of
$R$-equivalence classes of $K$-rational points \cite[Prop. 10]{ct-s}. One of the main results of
this paper is that \emph{any birational functor $\Phi$ is canonically a direct limit of functors
of the form $\Phi_K$}.

This follows from Theorem
\ref{T1} below via the complement to Yoneda's lemma (\cite[Exp. I, Prop. 3.4 p. 19]{SGA4} or
\cite[Ch. III, \S, Th. 1 p. 76]{maclane}). Here is the philosophy which led to this
result and others presented here:

Birational geometry over a field $F$ is the study of function fields over $F$, viewed as generic
points of algebraic varieties\footnote{By convention all varieties are irreducible here,
although not necessarily geometrically irreducible.}, or alternately the study of algebraic
$F$-varieties ``up to proper closed subsets". In this context, two ideas seem related:

\begin{itemize}
\item places between function fields;
\item rational maps.
\end{itemize}

The main motivation of this paper has been to understand the precise relationship between
them. We have done this by defining two rather different ``birational categories" and
comparing them.

The first idea gives the category $\place$ (objects: function fields; morphisms: $F$-places),
that we like to call the
\emph{coarse birational category}. For the second idea, one has to be a
little careful: the na\"{\i}ve attempt at taking as objects smooth
varieties and as morphisms rational maps does not work because,
as was pointed out to us by H\'el\`ene Esnault, one
cannot compose rational maps in general. On the other hand, one can
certainly start from the category $\Sm$ of smooth $F$-varieties and \emph{localise} it (in
the sense of Gabriel-Zisman 
\cite{gz}) with respect to the set $S_b$ of birational morphisms. We like to call the resulting category $S_b^{-1}\Sm$ the \emph{fine birational category}. 
By hindsight, the problem mentioned by Esnault can be
understood as a problem of calculus of fractions of $S_b$ in $\Sm$.

In spite of the lack of calculus of fractions, the category $S_b^{-1}\Sm$ was studied in
\cite{loc} and we were able to show that, under resolution of singularities, the natural
functor $S_b^{-1}\Sm^\proper\to S_b^{-1}\Sm$ is an equivalence of categories, where
$\Sm^\proper$ denotes the full subcategory of smooth proper varieties (\loccit, Prop. 8.5).

What was not done in \cite{loc} was the computation of Hom sets in $S_b^{-1}\Sm$. This is the
first main result of this paper:

\begin{Th}[\cf Th. \protect{\ref{t6.1} and Cor. \ref{c5.1}}]\label{T1}  Let $X,Y$ be two smooth $F$-varieties, with $Y$ proper. Then,\\
a) In
$S_b^{-1}\Sm$, we have an isomorphism
\[\Hom(X,Y) \simeq Y(F(X))/R\]
where the right hand side is the set of $R$-equivalence classes in the sense of Manin.\\
b) The natural functor
\[S_b^{-1} \Sm_*^\proper \to S_b^{-1} \Sm \]
is fully faithful. Here $\Sm_*^\proper$ is the full subcategory of $\Sm$ with objects those smooth proper varieties whose function field has a cofinal set of smooth proper models (see Definition \ref{d3.1}).
\end{Th}

\enlargethispage*{20pt}

For the link with the result mentioned at the beginning of the introduction, note that $\Sm_*^\proper=\Sm^\proper$ in characteristic $0$, and any
birational functor on smooth proper varieties factors uniquely through
$S_b^{-1}\Sm^\proper$, by the universal property of the latter category. 

Theorem \ref{T1} implies that $X\mapsto X(F)/R$ is a birational invariant of smooth proper varieties in any characteristic (Cor. \ref{c6.1}), a fact which seemed to be known previously only in characteristic $0$ \cite[Prop. 10]{ct-s}. It also implies that one can
define a composition law on classes of
$R$-equivalence (for smooth proper varieties), a fact which is not at all obvious a priori.

The second main result is a comparison between the coarse and fine birational categories. Let $\dv$ be the subcategory of $\place$ whose objects are separably generated function fields and morphisms are generated by field extensions and places associated to ``good'' discrete valuation rings (Definition \ref{d-good}).

\begin{Th}[\cf Th. \protect{\ref{t4.1} and \ref{tcor-loc}}]\label{T2} 
a) There is an equivalence of categories
\[\overline\Psi:(\dv/\h')^\op\iso S_b^{-1}\Sm\]
where $\dv/\h'$ is the quotient category of $\dv$ by the equivalence relation generated by two elementary relations: homotopy of places (definition \ref{deq1}) and ``having a common centre of codimension $2$ on some smooth model''.\\
b) If $\car F=0$, the natural functor $\dv/h'\to \place/\h''$ is an equivalence of categories, where $\h''$ is generated by homotopy of places and ``having a common centre on some smooth model".
\end{Th}

(See \S \ref{s1.2} for the notion of an equivalence relation on a category.)

Put together, Theorems \ref{T1} and \ref{T2} provide an answer to a question of Merkurjev: given a smooth proper variety $X/F$, give a purely birational description of the  set  $X(F)/R$. This answer is rather clumsy because the equivalence relation $\h'$ is not easy to handle; we hope to come back to this issue later.

\enlargethispage*{20pt}

Let us introduce the set $S_r$ of \emph{stable birational morphisms}: by definition, a morphism $s:X\to Y$ is in $S_r$ if it is dominant and the function field extension $F(X)/F(Y)$ is purely transcendental. We wondered about the nature of the localisation functor $S_b^{-1}\Sm\allowbreak\to S_r^{-1}\Sm$ for a long time, until the answer was given us by Colliot-Th\'el\`ene through a wonderfully simple geometric argument (see Appendix \ref{colliot}): 

\begin{Th}[\cf Th. \protect{\ref{sb=sr}}]\label{T3} 
 The functor $S_b^{-1}\Sm\to S_r^{-1}\Sm$ is an equivalence of categories.
\end{Th}

This shows a striking difference between birational functors and numerical birational invariants,  many of which are not stably birationally invariant (for example, plurigenera).

Theorems \ref{T1} and \ref{T2} are substantial improvements of our results in the first version of this paper \cite{Birat0}, which were proven only in characteristic $0$: even in characteristic $0$, Theorem \ref{T2} is new with respect to \cite{Birat0}. Their proofs are intertwined in a way we shall describe now.

The first point is to relate the coarse and fine birational categories, as 
there is no obvious comparison functor between them. There are two essentially different approaches to this question. In the first one:

\begin{itemize} 
\item We 
introduce (Definition
\ref{d1}) an
``incidence category" $\SmP$, whose objects are  smooth 
$F$-varieties and morphisms from $X$ to $Y$ are given by pairs $(f,\lambda)$,
where $f$ is a morphism $X\to Y$, $\lambda$ is a place $F(Y)\tto F(X)$ and
$f,\lambda$ are \emph{compatible} in an obvious sense. This category maps to both $\place^\op$
and $\Sm$ by obvious forgetful functors. Replacing $\Sm$ by $\SmP$ turns out to have a strong rigidifying effect. 
\item We embed $\place^\op$ in the category of locally ringed spaces via the ``Riemann-Zariski'' variety attached to a function field.
\end{itemize}

In this way, we obtain a naturally commutative diagram
 \[
\xymatrix{
&S_b^{-1}\Sm^\proper_* \P\ar[dr]^{\overline \Phi_2^*}\ar[dl]_{\overline \Phi_1^*}\\
\place_*^\op\ar[dr]^{\bar \Sigma} && S_b^{-1}\Sm^\proper_*\ar[dl]_{\bar J}\\
&S_b^{-1}\widehat{\Sm^\proper_*}
}
\]
where $\place_*$ denotes the full subcategory of $\place$ consisting of the function fields of varieties in $\Sm^\proper_*$ (compare Theorem \ref{T1}). Then  $J$ is an \emph{equivalence of categories}\footnote{So is $\overline\Phi_1^*$.} and the induced functor
\begin{equation}\label{psi*}
\Psi_*: \place_*^\op \to S_b^{-1}\Sm_*^{\proper}\tag{*}
\end{equation}
is \emph{full and essentially surjective} (Theorems \ref{t3.2} and \ref{cplsm}).

This is more or less where we were in the first version of this paper \cite{Birat0}, except for the use of the categories $\Sm_*$ and $\place_*$ which allow us to state results in any characteristic; in \cite{Birat0}, we also proved Theorem \ref{T1} when $\car F=0$, using resolution of singularities and a complicated categorical method.\footnote{Another way to prove Theorem \ref{T1} in characteristic $0$, which was our initial method, is to define a composition law on $R$-equivalence classes by brute force  (still using resolution of singularities) and to proceed as in the proof of Proposition \ref{leq1}.} 

The second approach is to construct a functor $\dv^\op\to S_b^{-1}\Sm$ directly. Here the new and decisive input is the recent paper of Asok and Morel \cite{a-m}, and especially the results of its \S 6: they got the insight that, working with discrete valuations of rank $1$, all the resolution that is needed  is ``in codimension $2$''. We implement their method in \S \ref{HR} of the present paper, which leads to a rather simple proof of Theorems \ref{T1} and \ref{T2} in any characteristic. Another key input is a recent uniformisation theorem of Knaf and Kuhlmann \cite{knaf-kuhlmann}.

Let us now describe the contents in more detail. We start by setting up notation in Section \ref{s1}, which ends with Theorem \ref{T3}.  
In Section
\ref{s3}, we introduce the incidence category $\SmP$ sitting in the larger category
$\VarP$, the forgetful functors $\VarP\to \Var$ and $\VarP\to \place^\op$, and
prove elementary results on these functors (see Lemmas \ref{l6} and \ref{l2bis}). 
In Section \ref{pvr}, 
 we endow the abstract Riemann variety with the structure of a locally ringed space, and 
prove that it is a cofiltered inverse limit of proper models, viewed as schemes (Theorem \ref{plim}): this ought to be well-known but we couldn't find
a reference. 
We apply these results to construct in \S\ref{s4} the functor \eqref{psi*}, using calculus of fractions. In section \ref{eqcat}, we study calculus of fractions in greater generality; in particular, we obtain a partial calculus of fractions in $S_b^{-1}\Sm_*$ in Proposition \ref{p3.4}. 

In \S\ref{HR}, we introduce a notion of homotopy on $\place$ and  the subcategory $\dv$.  We then relate our approach to the work of Asok-Morel
\cite{a-m} to prove Theorems \ref{T1} and \ref{T2}. We make the link between the first and second approaches in Theorem \ref{tcor-loc} = Theorem \ref{T2} b).

Section \ref{sbr} discusses variants of Koll\'ar's notion of rational chain connectedness (which goes back to Chow under the name of linear connectedness), recalls classical theorems of Murre, Chow and van der Waerden, states new theorems of Gabber including the one proven in Appendix \ref{gabber}, and draws some consequences in Theorem \ref{c3.5}. Section \ref{s7} discusses some applications, among which we like to mention the existence of a ``universal birational quotient'' of the fundamental group of a smooth variety admitting a smooth compactfication (\S \ref{kan}). We finish with a few open questions in \S \ref{s7.7}.

This paper grew out of the preprint \cite{birat}, where some of its results were initially
proven. We decided that
the best was to separate the present results, which have little to do with motives, from the
rest of that work.  Let us end with a word on the relationship between $S_b^{-1}\Sm$ and the $\A^1$-homotopy category of schemes $\sH$ of Morel-Voevodsky \cite{mv}.  One of the main results of Asok and Morel in \cite{a-m} is a proof of the following conjecture of  Morel in the proper case (\loccit Th. 2.4.3):

\begin{Conj}[\protect{\cite[p. 386]{morel}}] If $X$ is a smooth variety, the natural map
\[X(F)\to \Hom_{\sH}(\Spec F,X)\]
is surjective and identifies the right hand side with the quotient of the set $X(F)$ by the
equivalence relation generated by
\[(x\sim y) \iff \exists h:\A^1\to X\mid h(0) = x \text{ and } h(1)=y.\]
\end{Conj}

(Note that this ``$\A^1$-equivalence" coincides with $R$-equivalence if $X$ is proper.) Their result can then be enriched as follows:

\begin{Th}[\protect{\cite{ck}}]
The Yoneda embedding of  $\Sm$ into the category of simplicial presheaves of sets on $\Sm$ induces
a fully faithful functor
\[S^{-1}_b\Sm\To S_b^{-1}\sH \]
where $S_b^{-1}\sH$ is a suitable localisation of $\sH$ with respect to birational morphisms.
\end{Th}

\subsection*{Acknowledgements} We would like to thank the many colleagues who gave us insights about this work, among whom  Jean-Louis 
Colliot-Th\'e\-l\`e\-ne, H\'e\-l\`e\-ne Es\-nault, Najmuddin Fakhruddin, Ofer Gabber, Hagen Knaf, Geor\-ges Maltsiniotis, Vi\-kram Mehta, Bernard Teissier, Michael Temkin, I-Hsun Tsai and Michel Va\-qui\'e.  Finally, we thank Colliot-Th\'el\`ene and Gabber for kindly allowing us to include Appendices
\ref{colliot} and \ref{gabber} in this paper.

\subsection*{Conventions} $F$ is the base field. ``Variety'' means irreducible 
separated $F$-scheme of finite type. All morphisms are $F$-morphisms. If $X$ is a variety,
$\eta_X$ denotes its generic point.

\section{Preliminaries and notation}\label{s1}

In this section, we collect some basic material that will be used in the
paper. This allows us to fix our notation.

\subsection{Localisation of categories and calculus of fractions}
\label{cf}

We refer to Gabriel-Zisman \cite[Chapter I]{gz} for the necessary
background. Recall \cite[I.1]{gz} that if $\sC$ is a small category and
$S$ is a collection of morphisms in $\sC$, there is a category 
$\sC[S^{-1}]$ and a  functor $\sC\to
\sC[S^{-1}]$ which is universal among functors from $\sC$ which invert the
elements of $S$. When $S$ satisfies \emph{calculus of fractions}
\cite[I.2]{gz} the category $\sC[S^{-1}]$ is equivalent to another one,
denoted $S^{-1}\sC$ by Gabriel and Zisman, in which the Hom sets are more
explicit.


If $\sC$ is only essentially small, one can construct a category verifying the same $2$-universal property by starting from an equivalent small category, provided $S$ contains the identities. All categories considered in this paper are subcategories of $\Var(F)$ (varieties over our base field $F$) or $\place(F)$ (finitely generated extensions of $F$, morphisms given by places), hence are essentially small. 

We shall encounter situations where calculus of fractions
is satisfied, as well as others where it is not. We shall take
the practice to abuse notation and write $S^{-1}\sC$ rather than
$\sC[S^{-1}]$ even when calculus of fractions is not verified.

\begin{nota}\label{no1.1} If $(\sC,S)$ is as above, we write $\langle S\rangle$ for the \emph{saturation} of $S$: it is the set of morphisms $s$ in $\sC$ which become invertible in $S^{-1}\sC$. We have $S^{-1} \sC=\langle S\rangle^{-1}\sC$ and $\langle S\rangle$ is maximal for this property.
\end{nota}

Note the following easy lemma:

\begin{lemma}\label{lfess} Let $T:\sC\to \sD$ be a full and essentially surjective functor. Let $S\in Ar(\sC)$ be a set of morphisms. Then the induced functor $\bar T:S^{-1} \sC\to T(S)^{-1} \sD$ is full and essentially surjective.
\end{lemma}

\begin{proof} Essential surjectivity is obvious. Given two objects $X,Y\in S^{-1} \sC$, a morphism from $\bar T(X)$ to $\bar T(Y)$ is given by a zig-zag of morphisms of $\sD$. By the essential surjectivity of $T$, lift all vertices of this zig-zag, then lift its edges thanks to the fullness of $T$.
\end{proof}

\subsection{Equivalence relations}\label{s1.2}

\begin{defn}\label{d1.5} Let $\sC$ be a category. An
\emph{equivalence relation} on $\sC$ consists, for all
$X,Y\in \sC$, of an equivalence relation $\sim_{X,Y}=\sim$ on
$\sC(X,Y)$ such that $f\sim g$ $\Rightarrow$ $fh\sim gh$
and $kf\sim kg$ whenever it makes sense.
\end{defn}

In \cite[p. 52]{maclane}, the above notion is called a `congruence'.
Given an equivalence relation $\sim$ on $\sC$, we may form the
\emph{factor category} $\sC/\sim$, with the same objects as $\sC$
and such that $(\sC/\sim)(X,Y)=\sC(X,Y)/\sim$. This category and the
projection functor $\sC\to \sC/\sim$ are universal for functors
from $\sC$ which equalise equivalent morphisms.

\begin{ex} Let $\sA$ be an Ab-category (sets of morphisms are
abelian groups and composition is bilinear). An \emph{ideal}
$\sI$ in $\sA$ is given by a subgroup $\sI(X,Y)\subseteq \sA(X,Y)$
for all $X,Y\in \sA$ such that $\sI\sA\subseteq \sI$ and
$\sA\sI\subseteq \sI$. Then the ideal $\sI$ defines an equivalence
relation on $\sA$, compatible with the additive structure.
\end{ex}

Let $\sim$ be an equivalence relation on the category $\sC$. We
have the collection $S_\sim=\{f\in \sC\mid f$ is invertible in
$\sC/\sim\}$. The functor $\sC\to \sC/\sim$ factors as a functor
$S_\sim^{-1}\sC\to \sC/\sim$. Conversely, let $S\subset \sC$ be a set
of morphisms. We have the equivalence relation $\sim_S$ on $\sC$
such that $f\sim_S g$ if $f=g$ in $S^{-1}\sC$, and the localisation
functor $\sC\to S^{-1}\sC$ factors as $\sC/\sim_S\to S^{-1}\sC$.
Neither of these two factorisations is an equivalence of categories
in general; however, \cite[Prop. 1.3.3]{GP} remarks that if $f\sim g$ implies $f = g$ in
$S_\sim^{-1}\sC$, then $S_\sim^{-1}\sC\to \sC/\sim$ is an isomorphism of categories.

\begin{exo} Let $A$ be a commutative ring and $I\subseteq A$ an
  ideal. 

a) Assume that the set of minimal primes of $A$ that do not contain
$I$ is finite (\eg that $A$ is noetherian). Show that the following
two conditions are equivalent:
\begin{thlist}
\item There exists a multiplicative subset $S$ of $A$ such that
  $A/I\allowbreak \simeq
  S^{-1}A$ (compatibly with the maps $A\to A/I$ and $A\to S^{-1}A$). 
\item $I$ is generated by an idempotent.
\end{thlist}

(Hint: show first that, without any hypothesis, (i) is equivalent to 
\begin{thlist}
\item[(iii)] For any $a\in I$, there exists $b\in I$ such that
  $ab=a$.)
\end{thlist}

b) Give a counterexample to (i) $\Rightarrow$ (ii) in the general case
(hint: take $A=k^\N$, where $k$ is a field). 
\end{exo}

\subsection{Places, valuations and centres \protect\cite[Ch. VI]{zs},
\cite[Ch. 6]{Bour}}\label{place}

Recall \cite[Ch. 6, \S 2, Def. 3]{Bour} that a \emph{place} from a
field $K$ to a field $L$ is a map $\lambda:K\cup\{\infty\}\to
L\cup\{\infty\}$ such that $\lambda(1)=1$ and $\lambda$ preserves sum
and product whenever they are defined. We shall usually denote
places by screwdriver arrows:
\[\lambda:K\tto L.\]

Then $\sO_\lambda=\lambda^{-1}(L)$ is a valuation ring of $K$ and
$\lambda_{\mid\sO_\lambda}$ factors as
\[\sO_\lambda\surj \kappa(\lambda)\inj L\]
where $\kappa(\lambda)$ is the residue field of $\sO_\lambda$.
Conversely, the data of a valuation ring $\sO$ of $K$ with residue
field $\kappa$ and of a field homomorphism $\kappa\to L$ uniquely
defines a place from $K$ to $L$ (\loccit, Prop. 2). It is easily
checked that the composition of two places is a place.

\begin{caution}  Unlike Zariski-Samuel \cite{zs} and other authors \cite{mva,knaf-kuhlmann}, we compose places in the same order as extensions of fields: so if $K\overset{\lambda}{\tto} L \overset{\mu}{\tto} M$ are two successive places, their composite is written $\mu\lambda$ in this paper. We hope this will not create confusion.
\end{caution}

If $K$ and $L$ are extensions of $F$, we say that $\lambda$ is
an \emph{$F$-place} if $\lambda_{\mid F}=Id$ and then write
$F(\lambda)$ rather than $\kappa(\lambda)$.

In this situation, let $X$ be an integral $F$-scheme of finite type
with function field $K$. A point $x\in X$ is a \emph{centre} of a
valuation ring $\sO\subset K$ if $\sO$ dominates the local ring
$\sO_{X,x}$. If $\sO$ has a centre on $X$, we sometimes say that
$\sO$ is \emph{finite} on $X$. As a special case of the valuative
criterion of separatedness (\resp of the valuative criterion of
properness), $x$ is unique (\resp and exists) for all $\sO$ if and only
if
$X$ is separated (\resp proper) \cite[Ch. 2, Th. 4.3 and
4.7]{hartshorne2}.

On the other hand, if $\lambda:K\tto L$ is an $F$-place, then a point
$x\in X(L)$ is a \emph{centre} of $\lambda$ if there is a map
$\phi:\Spec \sO_\lambda\to X$ letting the diagram
\[\xymatrix{
\Spec \sO_\lambda\ar[dr]^\phi&\Spec K\ar[l]\ar[d]\\
\Spec L\ar[u]^{\lambda^*}\ar[r]^x& X
}\]
commute. Note that the image of the closed point by $\phi$ is then a centre of the
valuation ring $\sO_\lambda$ and that $\phi$ uniquely determines $x$.

In this paper, when $X$ is separated we shall denote by $c_X(v)\in X$
the centre of a valuation $v$ and by $c_X(\lambda)\in X(L)$ the
centre of a place $\lambda$, and carefully distinguish between the two notions
(one being a scheme-theoretic point and the other a rational point).

We have the following useful lemma from Vaqui\'e \cite[Prop. 2.4]{mva}; we reproduce its proof.

\begin{lemma}\label{lvaquie} Let $X\in \Var$, $K=F(X)$, $v$ a valuation on $K$ with residue field $\kappa$ and $\bar v$ a valuation on $\kappa$. Let  $v'= \bar v \circ v$ denote the composite valuation.\\ 
a) If $v'$ is finite on $X$, so is $v$.\\
b) Assume that $v$ is finite on $X$, and let $Z\subset X$ be the closure of its centre (so that $F(Z)\subseteq \kappa$).  Then $v'$ is finite on $X$  if and only if  [the restriction to $F(Z)$ of] $\bar v$ is finite on $Z$, and then $c(\bar v)\in Z$ equals $c(v')\in X$.
\end{lemma} 

\begin{proof} We may assume that  $X = \Spec A$ is an affine variety. Denoting respectively by $V , V', \bar V$ and $\fm, \fm', \bar \fm$ the valuation rings associated to $v,v',\bar v$ and their maximal ideals, we have $(0)\subset\fm\subset\fm' \subset V' \subset V \subset K$ and $\bar \fm \subset \bar V = V'/\fm \subset \bar K = V/\fm$. 

a) $v'$ is finite on $X$ if and only if $A\subset V'$, which implies $A\subset V$.

b) The centres of the valuations $v$ and $v'$ on $X$ are defined by the prime ideals $\fp=A\cap\fm$ and $\fp' =A\cap\fm'$ of $A$, and the centre of the valuation $\bar v$ on $Z = \Spec \bar A$, with $\bar A = A/\fp$ is defined by the prime ideal $\bar \fp = \bar A \cap \bar\fm$ of $\bar A$. Then the claim is a consequence of the equality $\bar\fp = \fp'/\fp$.
\end{proof}

\subsection{Rational maps}\label{ratmaps} Let $X,Y$ be two $F$-schemes
of finite type, with $X$ integral and $Y$ separated. Recall that a \emph{rational map} from $X$ to $Y$ is a pair
$(U,f)$ where $U$ is a
dense open subset of $X$ and $f:U\to Y$ is a morphism. Two rational
maps $(U,f)$ and $(U',f')$ are equivalent if there exists a dense open
subset $U''$ contained in
$U$ and $U'$ such that $f_{\mid U''}=f'_{\mid U''}$. We denote by
$\Rat(X,Y)$ the set of equivalence classes of rational maps, so that
\[\Rat(X,Y)=\colim \Map_F(U,Y)\]
where the limit is taken over the open dense subsets of $X$. There is a largest open
subset $U$ of $X$ on which a given rational map $f:X\ttto Y$ is
defined \cite[Ch. I, Ex. 4.2]{hartshorne2}. The (reduced) closed
complement $X-U$ is called the \emph{fundamental set} of $f$
(notation: $\Exc(f)$). We say that $f$ is \emph{dominant} if $f(U)$ is
dense in $Y$.

Similarly, let $f:X\to Y$ be a birational morphism. The complement of
the largest open subset of $X$ on which $f$ is an isomorphism is
called the \emph{exceptional locus} of $f$ and is denoted by
$\Fund(f)$. \label{exc} 

Note that the sets $\Rat(X,Y)$ only define a \emph{precategory}
(or diagram, or diagram scheme, or quiver) $\Rat(F)$, because rational
maps cannot be composed in general. To clarify this, let $f:X\ttto Y$
and $g:Y\ttto Z$ be
two rational maps, where $X,Y,Z$
are varieties. We say that $f$ and $g$ are \emph{composable} if
$f(\eta_X)\notin \Exc(g)$, where
$\eta_X$ is the generic point of $X$. Then there exists an open subset
$U\subseteq X$ such that $f$ is defined on $U$ and $f(U)\cap
\Exc(g)=\emptyset$, and $g\circ f$ makes sense as a rational map. This
happens in two important cases:
\begin{itemize}
\item $f$ is dominant;
\item $g$ is a morphism.
\end{itemize}

This composition law is associative wherever it makes sense. In
particular, we do have the category $\Rat_\dom(F)$ with objects 
$F$-varieties and morphisms dominant rational maps. Similarly,
the category $\Var(F)$ of \ref{zoo1} acts on $\Rat(F)$ on the left. 

\begin{lemma}[\protect{\cite[Lemma 8.2]{loc}}]\label{l1} Let $f,g:X\to
  Y$ be two morphisms, with $X$ 
integral and $Y$ separated. Then $f=g$ if and only if
$f(\eta_X)=g(\eta_X)=:y$ and $f,g$ induce the same map $F(y)\to F(X)$ on
the residue fields.\qed
\end{lemma}

For $X,Y$ as above, there is a well-defined map
\begin{align}
\Rat(X,Y)&\to Y(F(X))\label{eq2.1}\\
(U,f)&\mapsto f_{\mid \eta_X}\notag
\end{align}
where $\eta_X$ is the generic point of $X$.

\begin{lemma}\label{l2.1} The map \eqref{eq2.1}  is bijective.
\end{lemma}

\begin{proof} Surjectivity is clear, and injectivity follows from
Lemma \ref{l1}.\end{proof}

\subsection{The graph trick} \label{p1.1} We shall often use this
well-known and basic device, which allows us to replace a rational map by
a morphism.
 
Let $U,Y$ be two $F$-varieties. Let $j:U\to X$ be an open immersion
($X$ a variety) and $g:U\to Y$ a morphism. Consider the graph 
$\Gamma_g\subset U\times Y$. By the first projection, $\Gamma_g\iso U$.
Let $\bar\Gamma_g$ be the closure of $\Gamma_g$ in $X\times Y$, viewed as
a reduced scheme. Then 
the rational map $g:X\ttto Y$ has been replaced by $g':\bar\Gamma_g\to Y$
(second projection) through the birational map $p:\bar\Gamma_g\to X$
(first projection). Clearly, if $Y$ is
proper then $p$ is proper. 

\subsection{Structure theorems on varieties} Here we collect two
well-known results, for future reference.

\begin{thm}[Nagata \protect{\cite{nagata}}]\label{nagata} Any variety
  $X$ can be embedded into a proper
  variety $\bar X$. We shall sometimes call $\bar X$ a
  \emph{compactification} of $X$.
\end{thm}


\begin{thm}[Hironaka \protect{\cite{hironaka}}]\label{Hironaka} If
  $\car F=0$,\\ 
a) For any
  variety $X$ there exists a projective birational morphism $f:\tilde X\to
  X$ with $\tilde X$ smooth. (Such a morphism is sometimes called a
  modification.) Moreover, $f$ may be chosen such that it is an
  isomorphism away from the inverse image of the singular locus of
  $X$. In particular, any smooth variety $X$ may be embedded as an
  open subset of a smooth proper variety (projective if $X$ is
  quasi-projective).\\
b) For any proper birational morphism $p:Y\to X$ between smooth
varieties, there exists a
proper birational morphism $\tilde p:\tilde Y\to X$ which factors
through $p$ and is a composition of blow-ups with smooth centres.
\end{thm}

In some places we shall  assume characteristic $0$ in order to use resolution of
singularities. We shall specify this by putting an asterisk to the statement of the
corresponding result (so, the asterisk will mean that the characteristic $0$ assumption is due
to the use of Theorem \ref{Hironaka}).

\subsection{Some multiplicative systems}\label{zoo1} Let $\Intsep(F)=\Intsep$ be the category of \emph{$F$-varieties}: objects
are $F$-varieties (\ie integral separated $F$-schemes of finite type) and
morphisms are all $F$-morphisms. We write $\Sm(F)=\Sm$ for its full subcategory consisting of smooth varieties. As in \cite{loc}, the superscripts $^\qp,^\proper, ^\proj$ respectively mean quasi-projective, proper and projective.

As in \cite{loc}, we shall use various collections of morphisms of
$\Var$ that are to be inverted:
\begin{itemize}
\item \emph{Birational morphisms} $S_b$: $s\in S_b$ if $s$ is dominant
  and induces an isomorphism of function fields.
\item \emph{Stably birational morphisms} $S_r$: $s\in S_r$ if $s$ is
  dominant and induces a purely transcendental extension of function fields.
\end{itemize}

In addition, we shall use the following subsets of $S_b$:
\begin{itemize}
\item $S_o$: open immersions
\item $S_b^p$: proper birational morphisms
\end{itemize}
and of $S_r$:
\begin{itemize}
\item $S_r^p$: proper stably birational morphisms
\item $S_h$: the projections $pr_2:X\times\P^1\to X$.
\end{itemize}

We shall need the following lemma:

\begin{lemma}\label{sb=so} 
a)  In $\Var$ and $\Sm$, we have $\langle S_b\rangle = \langle S_o\rangle$ and $\langle S_r\rangle= \langle S_b\cup S_h\rangle$
(see Notation \ref{no1.1}).\\
b) We have  $\langle S_r^p\rangle =\langle S_b^p\cup
S_h\rangle$ in $\Var$, *and also in $\Sm$ under resolution of singularities.\\
\end{lemma}

\begin{proof} 

a)  The first equality is left to the reader. For the second one, given a morphism $s:Y\to X$ in $S_r$ with $X,Y\in\Var$ or $\Sm$, it suffices to
consider a
commutative diagram
\begin{equation}\label{eqblp}
\xymatrix{
&\tilde Y\ar[dl]_t\ar[dr]^u\\
Y\ar[dr]_s && X\times(\P^1)^n\ar[dl]^\pi\\
&X
}
\end{equation}
with $t,u\in S_o$, $\tilde Y$ a common open subset of $Y$ and $X\times(\P^1)^n$.

b) For a morphism $s:Y\to X$ in $S_r^p$ with $X,Y\in\Var$, we get a diagram \eqref{eqblp}, this time with $t,u\in S_b^p$ and $\tilde Y$ obtained by the graph trick. If $X,Y\in \Sm$, we use resolution to replace $\tilde Y$ by a smooth variety.
%
%
\end{proof}

Here is now the main result of this section.

 \begin{thm}\label{sb=sr} In $\Sm$, the sets $S_b$ and $S_r$
have the same saturation. *This is also true for $S_b^p$ and $S_r^p$  under resolution of singularities.\\ In particular, the obvious functor $S_b^{-1}\Sm\to S_r^{-1}\Sm$ is an equivalence of categories.
\end{thm}

\begin{proof} Let us prove that $S_h$ is
contained in the saturation of $S_b^p$, hence in the saturation of $S_b$. Let  $Y$ be smooth variety, and let $f:Y\times \P^1\to Y$ be the
first projection. We have to show that $f$ becomes 
invertible in $(S_b^p)^{-1}\Sm$. By Yoneda's lemma, it suffices to show
that $F(f)$ is invertible for any (representable) functor
$F:(S_b^p)^{-1}\Sm^\op\to\Sets$. 
This follows by taking the proof of Appendix \ref{colliot} and ``multiplying" it by $Y$. 

To get Theorem \ref{sb=sr}, we now apply Lemma \ref{sb=so} a) and b). (Applying b) is where resolution of singularities is required.)
\end{proof}

\begin{rk} Theorem \ref{sb=sr} is also valid in $\Var$, without resolution of singularities hypothesis (same proof). Recall however that the functor $S_b^{-1}\Sm\to S_b^{-1}\Var$ induced by the inclusion $\Sm\inj \Var$ is far from being fully faithful \cite[Rk. 8.11]{loc}.
\end{rk}

\section{Places and morphisms}\label{s3}

\subsection{The category of places}

\begin{defn}\label{d1.1} We denote by $\place(F)=\place$ the category with
objects finitely generated extensions of $F$ and morphisms
$F$-places. We denote by $\field(F)=\field$ the subcategory of $\place(F)$
with the same objects, but in which morphisms are $F$-homomorphisms of
fields. We shall sometimes call the latter \emph{trivial places}.
\end{defn}

\begin{rk} \label{r1.1} 
If $\lambda:K\tto L$ is a morphism in
$\place$, then its residue field $F(\lambda)$ is finitely
generated over $F$, as a subfield of the finitely generated field
$L$. On the
other hand, given a finitely generated extension $K/F$, there
exist valuation rings of $K/F$ with infinitely generated residue
fields as soon as $\trdeg(K/F)>1$, \cf
\cite[Ch. VI, \S 15, Ex. 4]{zs}.
\end{rk}

In this section, we relate the categories
$\place$ and $\Var$. We start with the main tool, which is the notion of
compatibility between a place and a morphism.

\subsection{A compatibility condition}

\begin{defn}\label{d1} Let $X,Y\in \Var$, $f:X\ttto Y$ a
rational map
and $v:F(Y)\tto F(X)$ a place. We say that $f$ and $v$ are
\emph{compatible} if
\begin{itemize}
\item $v$ is finite on $Y$ (\ie has a centre in $Y$).
\item The corresponding diagram
\[\begin{CD}
\eta_X@>v^*>> \Spec \sO_v\\
@V{}VV @V{}VV\\
U@>f>> Y
\end{CD}\]
commutes, where $U$ is an open subset of $X$ on which $f$ is defined.
\end{itemize}
\end{defn}

\begin{prop}\label{l2} Let $X,Y,v$ be as in Definition \ref{d1}. Suppose
that $v$ is finite on $Y$, and let
$y\in Y(F(X))$ be its centre. Then a rational map $f:X\ttto Y$ is compatible with
$v$ if and only if
\begin{itemize}
\item $y=f(\eta_X)$ and
\item the diagram of fields
\[\xymatrix{
F(v)\ar[dr]^v\\
F(y)\ar[u] \ar[r]^{f^*}& F(X)
}\]
commutes.
\end{itemize}
In particular, there is at most one such $f$.
\end{prop}

\begin{proof} Suppose $v$ and $f$ compatible. Then $y=f(\eta_X)$ because
$v^*(\eta_X)$ is the closed point of $\Spec\sO_v$. The commutativity of
the diagram then follows from the one in Definition \ref{d1}. Conversely,
if
$f$ verifies the two conditions, then it is obviously compatible with
$v$. The last assertion follows from Lemma \ref{l1}.\end{proof}

\begin{cor}\label{c1.1} a) Let $Y\in\Var$ and let
$\sO$ be a valuation ring of $F(Y)/F$ with residue field $K$ and
centre $y\in Y$. Assume that
$F(y)\iso K$. Then, for any rational map $f:X\ttto Y$ with $X$ integral, such
that
$f(\eta_X)=y$, there exists a unique place $v:F(Y)\tto F(X)$ with
valuation ring
$\sO$ which is compatible with $f$.\\
b) If $f$ is an immersion, the condition $F(y)\iso K$ is also necessary
for the existence of $v$.\\
c) In particular, let $f:X\ttto Y$ be a dominant rational
  map. Then $f$ is compatible with the
trivial place $F(Y)\inj F(X)$, and this place is the only one with which
$f$ is compatible.
\end{cor}

\begin{proof} This follows immediately from Proposition \ref{l2}.\end{proof}

\begin{prop}\label{p1} Let $f:X\to Y$, $g:Y\to Z$ be two morphisms
of varieties. Let $v:F(Y)\tto F(X)$ and $w:F(Z)\tto F(Y)$ be two
places. Suppose that $f$ and $v$ are compatible and that $g$ and
$w$ are compatible. Then $g\circ f$ and $v\circ w$ are compatible.
\end{prop}

\begin{proof} We first show that $v\circ w$ is finite on $Z$. By definition, the
diagram
\[\begin{CD}
\eta_Y@>w^*>> \Spec \sO_w\\
@V{}VV @V{}VV\\
\Spec \sO_v@>>> \Spec \sO_{v\circ w}
\end{CD}\]
is cocartesian. Since the two compositions
\[\eta_Y\stackrel{w^*}{\longrightarrow} \Spec\sO_w\to Z\]
and
\[\eta_Y\to \Spec\sO_v\to Y\stackrel{g}{\longrightarrow} Z\]
coincide (by the compatibility of $g$ and $w$), there is a unique induced
(dominant) map $\Spec \sO_{v\circ w}\to Z$. In the diagram
\[\begin{CD}
\eta_X@>v^*>>\Spec\sO_v@>>> \Spec \sO_{v\circ w}\\
@V{}VV @V{}VV @V{}VV\\
X@>f>>Y@>g>>Z
\end{CD}\]
the left square commutes by compatibility of $f$ and $v$, and the
right square commutes by construction. Therefore the big rectangle
commutes, which means that $g\circ f$ and $v\circ w$ are compatible.\end{proof}

\subsection{The category $\VarP$}

\begin{defn}\label{d2} We denote by $\VarP(F)=\VarP$ the following
category:
\begin{itemize}
\item Objects are  $F$-varieties.
\item Let $X,Y\in \VarP$. A morphism $\phi\in \VarP(X,Y)$ is a
pair $(\lambda,f)$ with $f:X\to Y$ a morphism, $\lambda:F(Y)\tto F(X)$ a place and
$\lambda,f$ compatible.
\item The composition of morphisms is given by Proposition \ref{p1}.
\end{itemize}
If $\sC$ is a full subcategory of $\Var$, we also denote by $\sC \P(F)=\sC\P$ the full subcategory of
$\VarP$ whose objects are in $\sC$.
\end{defn}

We now want to do an elementary study of the two forgetful functors
appearing in the diagram below:
\begin{equation}\label{l3}
\begin{CD}
\VarP@>\Phi_1>> \place^\op\\
@V{\Phi_2}VV\\
\Intsep.
\end{CD}
\end{equation}

Clearly, $\Phi_1$ and $\Phi_2$ are essentially surjective. Concerning
$\Phi_2$, we have the following partial result on its fullness:

\begin{lemma}\label{l6} Let $f:X\ttto Y$ be a rational map, with $X$
integral and $Y$ separated. Assume that $y=f(\eta_X)$
is a \emph{regular} point (\ie $A=\sO_{Y,y}$ is regular). Then there is a
place
$v:F(Y)\tto F(X)$ compatible with $f$.
\end{lemma}

\begin{proof} By Corollary \ref{c1.1} a), it is sufficient to
produce a valuation ring $\sO$ containing $A$ and with the same residue
field as $A$. 

The following construction is certainly classical. Let $\fm$ be the
maximal ideal of $A$ and let $(a_1,\dots,a_d)$ be a regular sequence
generating $\fm$, with $d=\dim A=\codim_Y y$. For 
$0\le i< j\le d+1$, let
\[A_{i,j}=(A/(a_j,\dots,a_d))_\fp\] 
where $\fp=(a_{i+1},\dots, a_{j-1})$ (for
$i=0$ we
invert no $a_k$, and for $j=d+1$ we mod out no $a_k$). Then, for any
$(i,j)$,
$A_{i,j}$ is a regular local ring of dimension $j-i-1$. In particular,
$F_i=A_{i,i+1}$ is the residue field of $A_{i,j}$ for any $j\ge i+1$. We
have
$A_{0,d+1}=A$ and there are obvious maps
\begin{align*}
A_{i,j}&\to A_{i+1,j}\quad\text{(injective)}\\
A_{i,j}&\to A_{i,j-1}\quad\text{(surjective).}
\end{align*}

Consider the discrete valuation $v_i$ associated to the discrete
valuation ring $A_{i,i+2}$: it defines a place, still denoted by $v_i$,
from $F_{i+1}$ to $F_i$. The composition of these places is a place $v$
from $F_d=F(Y)$ to $F_0=F(y)$, whose valuation ring dominates $A$ and
whose residue field is clearly $F(y)$.\end{proof}

\begin{rk} \label{rk2.1} In Lemma \ref{l6}, the assumption that $y$ is a
regular point is necessary. Indeed, take for $f$ a
closed immersion.
By \cite[Ch. 6, \S 1, Th. 2]{Bour}, there exists a valuation ring
$\sO$ of $F(Y)$ which dominates $\sO_{Y,y}$ and whose residue field
$\kappa$ is an algebraic extension of $F(y)=F(X)$. However
we cannot choose $\sO$ such that $\kappa=F(y)$ in general. The same
counterexamples as in \cite[Remark 8.11]{loc} apply (singular curves, 
the point $(0,0,\dots,0)$ on the affine cone $x_1^2+ x_2^2+\dots+x_n^2=
0$ over $\R$ for $n\ge 3$). 
\end{rk}

Now concerning $\Phi_1$, we have:

\begin{lemma}\label{l2bis} Let $X,Y$ be two varieties and
  $\lambda:F(Y)\tto F(X)$ a place. Assume that $\lambda$ is finite on $Y$.
Then there exists a unique rational map $f:X\ttto Y$ compatible with $\lambda$.
\end{lemma}

\begin{proof} Let $y$ be the centre of $\sO_\lambda$ on $Y$ and $V=\Spec R$ an
affine neighbourhood of $y$, so that $R\subset \sO_\lambda$, and let
$S$ be the image of $R$ in $F(\lambda)$.  Choose a finitely generated
$F$-subalgebra $T$ of $F(X)$ containing $S$, with quotient field $F(X)$. 
Then
$X'=\Spec T$ is an affine model of $F(X)/F$. The composition $X'\to
\Spec S\to V\to Y$ is then compatible with
$v$. Its restriction to a common open subset $U$ of $X$ and $X'$ defines
the desired map $f$. The uniqueness of $f$ follows from Proposition
\ref{l2}.\end{proof}

\begin{rk} Let $Z$ be a third variety and $\mu:F(Z)\tto F(Y)$ be
  another place, finite on $Z$; let $g:Y\ttto Z$ be the rational map
  compatible with $\mu$. If $f$ and $g$ are composable, then $g\circ
  f$ is compatible with $\lambda\circ \mu$: this follows easily from
  Proposition \ref{p1}. However it may well happen that $f$ and $g$
  are not composable. For example, assume $Y$ smooth. Given $\mu$,
  hence $g$ (that we suppose not to be a morphism), choose $y\in
  \Exc(g)$ and find a $\lambda$ with centre $y$, for example by the
  method in the proof of Lemma \ref{l6}. Then the rational map $f$
  corresponding to $\lambda$ has image contained in $\Exc(g)$. 
\end{rk}

We conclude this section with a useful lemma which shows that
places rigidify the situation very much.

\begin{lemma}\label{l4}
a) Let $Z,Z'$ be two models of a function field $L$, with $Z'$
separated, and $v$ a valuation of $L$ with centres $z,z'$ respectively on
$Z$ and $Z'$. Assume that there is a birational morphism $g:Z\to Z'$. Then
$g(z)=z'$.\\
b) Consider a diagram
\[\xymatrix{
&Z\ar[dd]^g\\
X\ar[ur]^f\ar[dr]_{f'}\\
&Z'
}\]
with $g$ a birational morphism. Let $K=F(X)$, $L=F(Z)=F(Z')$ and
suppose given a place $v:L\tto K$
compatible both with $f$ and $f'$. Then $f'=g\circ f$. 
\end{lemma}

\begin{proof} a) Let $f:\Spec \sO_v\to Z$ be the dominant map determined by $z$. Then
$f'=g\circ f$ is a dominant map $\Spec \sO_v\to Z'$. By the valuative
criterion of separatedness, it must correspond to $z'$. b) This
follows from a) and Proposition \ref{l2}.\end{proof}

\section{Places, valuations and the Riemann varieties}\label{pvr}

In this section, we give a second categorical relationship between
the idea of places and that of algebraic varieties. This leads us to
consider Zariski's ``abstract Riemann surface of a field" as a locally
ringed space. We start
by giving the details of this theory, as we could not find it elaborated in the
literature\footnote{Except for a terse allusion in \cite[0.6, p. 146]{hironaka}:
we thank Bernard Teissier for pointing out this reference.}. We remark however that the study of `Riemann-Zariski spaces'
has recently been revived by different authors independently (see \cite
{fk}, \cite{mt1}, \cite{mt2}, \cite{mva}).

\subsection{Strict birational morphisms}

It will be helpful to work here with the following notion of \emph{strict birational morphisms}:
\[\uS_b=\{s\in S_b \mid s\text{ induces an equality of function fields}\}\]

In fact, the difference between $S_b$ and $\uS_b$ is immaterial in view of the following

\begin{lemma} \label{sb=usb} Any birational morphism of (separated) varieties is the composition
of a strict birational morphism and an isomorphism.
\end{lemma}

\begin{proof} Let $s:X\to Y$ be a birational morphism. First assume $X$ and $Y$ affine, with
$X=\Spec A$ and $Y=\Spec B$. Let $K=F(X)$ and $L=F(Y)$, so that $K$ is the quotient field of
$A$ and $L$ is the quotient field of $B$. Let $s^*:L\iso K$ be the isomorphism induced by $s$.
Then $A'\iso A=s^*(A)$, hence $s$ may be factored as $X\by{s'} X'\by{u} Y$ with $X'=\Spec A'$,
where $s'$ is strict birational and $u$ is an isomorphism. In the general case, we may patch
the above construction (which is canonical) over an affine open cover $(U_i)$ of $Y$ and an
affine open cover of $X$ refining $(s^{-1}(U_i))$.
\end{proof}

\subsection{The Riemann-Zariski variety as a locally ringed space}

\begin{defn}\label{dring} We denote by $\sR(F)=\sR$ the full subcategory of
the category
of locally ringed spaces such that $(X,\sO_X)\in \sR$ if and only if
$\sO_X$ is a sheaf of local $F$-algebras.
\end{defn}

(Here, we understand by ``local ring" a commutative ring whose
non-invertible elements form an ideal, but we don't require it to be
Noetherian.)

\begin{lemma}\label{llim} Cofiltering inverse limits exist in $\sR$. More
precisely, if $(X_i,\sO_{X_i})_{i\in I}$ is a cofiltering inverse system
of
objects of $\sR$, its inverse limit is represented by $(X,\sO_X)$ with
$X=\lim X_i$ and $\sO_X=\colim p_i^*\sO_{X_i}$, where $p_i:X\to X_i$ is the
natural projection.
\end{lemma}

\begin{proof}[Sketch] Since a filtering direct limit of local rings for
local homomorphisms is local, the object of the lemma belongs to $\sR$ and
we are left to show that it satisfies the universal property of inverse
limits in $\sR$. This is clear on the space level, while on the sheaf level
it follows from the fact that inverse images of sheaves commute with
direct limits.
\end{proof}

Recall from Zariski-Samuel \cite[Ch. VI, \S 17]{zs} the
\emph{abstract Riemann surface} $S_K$ of a function field $K/F$: as a set,
it consists of all nontrivial valuations on $K$ which are trivial on $F$.
It is topologised by the following basis $\sE$ of open sets: if $R$ is a
subring of $K$, finitely generated over $F$, $E(R)\in\sE$ consists of all
valuations $v$ such that $\sO_v\supseteq R$.

As has become common practice, we slightly modify this definition:

\begin{defn} The \emph{Riemann variety} $\Sigma_K$ of $K$ is the following
ringed space:
\begin{itemize}
\item As a topological space, $\Sigma_K=S_K\cup\{\eta_K\}$ where $\eta_K$
is the trivial valuation of $K$. (The topology is defined as for $S_K$.) 
\item The set of sections over $E(R)$ of the structural sheaf of $\Sigma_K$ is
the intersection $\bigcap\limits_{v\in E(R)} \sO_v$, \ie the integral
closure of $R$.
\end{itemize}
\end{defn}

\begin{lemma}\label{lstalk} The stalk at $v\in \Sigma_K$ of the structure
sheaf is $\sO_v$. In particular, $\Sigma_K\in\sR$.
\end{lemma}

\begin{proof} Let $x_1,\dots,x_n\in \sO_v$. The subring $F[x_1,\dots,x_n]$ is finitely generated and
contained in $\sO_v$, thus $\sO_v$ is the filtering direct limit of the $R$'s
such that $v\in E(R)$.
\end{proof}

Let $R$ be a finitely generated $F$-subalgebra of $K$. We have a canonical
morphism of locally ringed spaces $c_R:E(R)\to \Spec R$ defined as
follows: on points we map $v\in E(R)$ to
its centre $c_R(v)$ on $\Spec R$. On the sheaf level, the map is defined by
the inclusions $\sO_{X,c_X(v)}\subset \sO_v$.

We now reformulate \cite[p. 115 ff]{zs} in scheme-theoretic language. Let
$X\in\Var$ be provided with a dominant
morphism $\Spec K\to X$ such that the corresponding field homomorphism
$F(X)\to K$ is an inclusion (as opposed to a monomorphism). We call such
an $X$ a \emph{Zariski-Samuel model} of $K$;
$X$ is a \emph{model} of $K$ if, moreover, $F(X)=K$. Note that
Zariski-Samuel models of $K$ form a cofiltering ordered set. Generalising
$E(R)$, we may define
$E(X)=\{v\in\Sigma_K\mid v\text{ is finite on } X\}$ for a Zariski-Samuel model of $K$;
this is still an open subset of $\Sigma_K$, being  the union of the $E(U_i)$,
where $(U_i)$ is some finite affine open cover of $X$. We still have a
morphism of locally ringed spaces $c_X:E(X)\to X$ defined by glueing the
affine ones. If $X$ is proper, $E(X)=\Sigma_K$ by the valuative criterion
of properness. Then:

\begin{thm}[Zariski-Samuel]\label{tzariski} 
The
induced morphism of ringed spa\-ces
\[\Sigma_K\to\lim X\]
where $X$ runs through the proper Zariski-Samuel models of $K$, is an
isomorphism in $\sR$. The generic point $\eta_K$ is dense in $\Sigma_K$.\\
\end{thm}

\begin{proof} Zariski and Samuel's
theorem \cite[th. VI.41 p. 122]{zs} says that the underlying morphism of
topological spaces is a homeomorphism; thus, by Lemma
\ref{llim}, we only need
to check that the structure sheaf of $\Sigma_K$ is the direct limit of the
pull-backs of those of the $X$. This amounts to showing that, for
$v\in\Sigma_K$, $\sO_v$ is the direct limit of the $\sO_{X,c_X(v)}$.

We argue essentially as in \cite[pp. 122--123]{zs} (or as in the proof of
Lemma \ref{lstalk}). Let
$x\in \sO_v$, and let $X$ be the projective Zariski-Samuel model determined by
$\{1,x\}$ as in \loccit, bottom p. 119, so that either $X\simeq
\P^1_F$ or $X=\Spec F'$ where $F'$ is a finite extension of $F$
contained in $K$. In both cases, $c=c_X(v)$ 
actually belongs to $\Spec F[x]$ and $x\in \sO_{X,c}\subset \sO_v$. 

Finally, $\eta_K$ is contained in every basic open set, therefore is dense
in $\Sigma_K$.
\end{proof}

\begin{defn}\label{dring1} Let $\sC$ be a full subcategory of $\Var$. We
denote by
$\hat{\sC}$ the full subcategory of $\sR$ whose objects
are cofiltered inverse limits of objects of $\sC$ under morphisms of
$\uS_b$
(\cf \S \ref{zoo1}). The natural inclusion $\sC\subset \hat\sC$ is denoted by $J$.
\end{defn}

Note that, for any function field $K/F$,
$\Sigma_K\in \widehat{\Var^\proper}$ by Theorem \ref{tzariski}. Also, for any $X\in \widehat{\Var}$, the function field $F(X)$ is well-defined. 

\begin{lemma}\label{lbasis} Let $X\in \widehat{\Var}$ and $K=F(X)$. \\
a) For a
finitely generated $F$-algebra
$R\subset K$, the set
\[E_X(R)=\{x\in X\mid R\subset O_{X,x}\}\]
is an open subset of $X$. These open subsets form a basis for the topology
of $X$.\\
b) The generic point $\eta_K\in X$ is dense in $X$, and $X$ is
quasi-compact.
\end{lemma}

\begin{proof} a) If $X$ is a variety, then $E_X(R)$ is open, being the set
of definition of the rational map $X\ttto\Spec R$ induced by the inclusion
$R\subset K$. In general, let $(X,\sO_X)=\lim_\alpha
(X_\alpha,\sO_{X_\alpha})$ with the $X_\alpha$ varieties and let
$p_\alpha:X\to X_\alpha$ be the projection. Since $R$ is finitely
generated, we have \[E_X(R)=\bigcup_\alpha p_\alpha^{-1}(E_{X_\alpha}(R))\]
which is open in $X$. 

Let $x\in X$: using Lemma \ref{llim}, we can find an $\alpha$ and an affine
open $U\subset X_\alpha$ such that $x\in p_\alpha^{-1}(U)$. Writing
$U=\Spec R$, we see that $x\in E_X(R)$, thus the $E_X(R)$ form a basis of
the topology of $X$.

In b), the density follows from a) since clearly $\eta_K \in E_X(R)$ 
for every $R$. The space $X$ is a limit of spectral spaces under spectral 
maps, and hence quasi-compact. Alternately, $X$ is compact in the constructible 
topology as compactness is preserved under inverse limits, and hence quasi-compact 
in the weaker Zariski topology.
\end{proof}


We are grateful to M. Temkin for pointing out an error in 
our earlier proof of quasi-compactness and providing the proof of b) above.

\begin{thm}\label{plim} Let $X=\lim X_\alpha$, $Y=\lim Y_\beta$ be two
objects of
$\widehat{\Var}$. Then we have a canonical isomorphism
\[\widehat{\Var}(X,Y)\simeq \lim_\beta\colim_\alpha \Var(X_\alpha,Y_\beta).\]
\end{thm}

\begin{proof} Suppose first that $Y$ is constant. We then have an obvious
map
\[\colim_\alpha \Var(X_\alpha,Y)\to \widehat{\Var}(X,Y).\]

Injectivity follows from Lemma \ref{l1}. For surjectivity, let $f:X\to Y$
be a morphism. Let $y=f(\eta_K)$. Since $\eta_K$ is dense in $X$ by Lemma
\ref{lbasis} b), $f(X)\subseteq \overline{\{y\}}$. This reduces us to the
case where $f$ is \emph{dominant}.

Let $x\in X$ and $y=f(x)$. Pick an affine open neighbourhood $\Spec R$ of
$y$ in $Y$. Then $R\subset \sO_{X,x}$, hence $R\subset
\sO_{X_\alpha,x_\alpha}$ for some $\alpha$, where $x_\alpha=p_\alpha(x)$, 
$p_\alpha:X\to X_\alpha$ being the canonical projection.
This shows that the rational
map $f_\alpha:X_\alpha\ttto Y$ induced by restricting $f$ to the generic
point is defined at $x_\alpha$ for $\alpha$ large enough.

Let $U_\alpha$ be the set of definition of $f_\alpha$. We have just shown
that $X$ is the increasing union of the open sets 
$p_\alpha^{-1}(U_\alpha)$. Since $X$ is
quasi-compact, this implies that $X=p_\alpha^{-1}(U_\alpha)$ for some
$\alpha$, \ie that $f$ factors through $X_\alpha$ for this value of
$\alpha$.

In general we have
\[\widehat{\Var}(X,Y)\iso \lim_\beta\widehat{\Var}(X,Y_\beta)\]
by the universal property of inverse limits, which completes the proof.
\end{proof}

\begin{rk} Let $\pro[\uS_b]\Var$ be the full subcategory of
the
category of pro-objects of $\Var$ consisting of the $(X_\alpha)$ in
which the transition maps $X_\alpha\to X_\beta$ are strict birational
morphisms. Then Theorem \ref{plim} may be reinterpreted as saying that
the functor
\[\lim:\pro[\uS_b]\Var\to \widehat{\Var}\]
is an \emph{equivalence of categories}.
\end{rk}

\subsection{Riemann varieties and places}\label{s3.3}

We are going to study two functors
\begin{gather*}
\Spec:\field^\op\to\widehat{\Var}\\
\Sigma:\place^\op\to\widehat{\Var}
\end{gather*}
and a natural transformation $\eta:\Spec\Rightarrow\Sigma\circ \iota$, where $\iota$ is the embedding $\field^\op\inj \place^\op$.

The first functor is simply $K\mapsto\Spec K$. The second one
maps $K$ to the Riemann variety $\Sigma_K$. Let $\lambda:K\tto L$ be an
$F$-place. We define $\lambda^*:\Sigma_L\to\Sigma_K$ as follows: if $w\in
\Sigma_L$, we may consider the associated place $\tilde w:L\tto F(w)$; then
$\lambda^*w$ is the valuation underlying $\tilde w\circ \lambda$. 

Let $E(R)$ be a basic open subset of $\Sigma_K$. Then 
\[(\lambda^*)^{-1}(E(R))=
\begin{cases}
\emptyset&\text{if $R\nsubseteq \sO_\lambda$}\\
E(\lambda(R))&\text{if $R\subseteq \sO_\lambda$.}
\end{cases}
\]

Moreover, if $R\subseteq \sO_\lambda$, then $\lambda$ maps
$\sO_{\lambda^*w}$ to $\sO_w$ for any valuation $w\in (\lambda^*)^{-1}E(R)$. This shows that $\lambda^*$ is continuous and defines a morphism of locally
ringed spaces. We leave it to the reader to check that $(\mu\circ
\lambda)^*=\lambda^*\circ\mu^*$.

Note that we have for any $K$ a morphism of ringed spaces
\begin{equation}\label{eqeta}
\eta_K:\Spec K\to\Sigma_K
\end{equation}
with image the trivial valuation of $\Sigma_K$ (which is its generic
point). This defines the natural transformation $\eta$ we alluded to.

\begin{prop}\label{psigma} The functors $\Spec$ and $\Sigma$ are fully
faithful;
moreover, for any $K,L$, the map
\[\widehat{\Var}(\Sigma_L,\Sigma_K)
\by{\eta_L^*}\widehat{\Var} (\Spec L,\Sigma_K)\]
is bijective.
\end{prop}

\begin{proof} The case of $\Spec$ is obvious. For the rest, let
$K,L\in\place(F)$ and consider the composition
\[\place(K,L)\overset{\Sigma}{\longrightarrow}
\widehat{\Var}(\Sigma_L,\Sigma_K)\by{\eta_L^*}\widehat{\Var}(\Spec
L,\Sigma_K).\]

It suffices to show that $\eta_L^*$ is injective and $\eta_L^*\circ \Sigma$ is bijective.

Let $\psi_1,\psi_2\in
\widehat{\Var}(\Sigma_L,\Sigma_K)$ be such that $\eta_L^*\psi_1=\eta_L^*\psi_2$.
Pick a proper model $X$ of $K$; by Theorem \ref{plim}, $c_X\circ \psi_1$
and $c_X\circ \psi_2$ factor through morphisms $f_1,f_2:Y\to X$ for
some model $Y$ of $L$. By Lemma 
\ref{l1}, $f_1=f_2$, hence $c_X\circ\psi_1=c_X\circ\psi_2$ and finally
$\psi_1=\psi_2$ by Theorem \ref{tzariski}.  Thus $\eta_L^*$ is injective.

On the other hand, let $\phi\in \widehat{\Var}(\Spec L,\Sigma_K)$ and
$v=\phi(\Spec L)$: then $\phi$ 
induces a homomorphism $\sO_v\to L$, hence a place
$\lambda:K\tto L$ and clearly $\phi =\eta_L^*\circ\Sigma(\lambda)$. This is the only place mapping to $\phi$. This
shows that the composition $\eta_L^*\circ\Sigma$
is bijective,  which
concludes the proof.
\end{proof}

\section{Two equivalences of categories}\label{s4}

In this section, we compare the localised categories $S_r^{-1}\place$ and a suitable version of $S_b^{-1}\Sm^\proper$ by using the techniques of the previous section.  First, we construct a full and essentially surjective functor 
\[\place_*^\op\to S_b^{-1} \Sm_*^\proper\]
in Corollary \ref{cplsm}, where $\Sm_*^\proper$ is the full subcategory of $\Sm$ formed of smooth varieties having a cofinal system of smooth proper models, and $\place_*\subseteq \place$ is the full subcategory of their function fields. Next, we prove in Theorem \ref{t3.2} that a suitable version of the functor $\Phi_1$ of \eqref{l3} becomes an equivalence of categories after we invert birational morphisms.

\subsection{The basic diagram} 
We start from the commutative diagram of functors 
\begin{equation}\label{eq4.4a}
\xymatrix{
&\VarP\ar[dr]^{\Phi_2}\ar[dl]_{\Phi_1}\\
\place^\op\ar[dr]^\Sigma && \Var\ar[dl]_J\\
&\widehat{\Var}
}
\end{equation}
where $\Phi_1$, $\Phi_2$ are the two forgetful functors of \eqref{l3}. Note that $\Sigma$ takes values in $\widehat{\Var^\proper}$, so this diagram restricts to a similar diagram where $\Var$ is replaced by $\Var^\proper$.

We can extend the birational morphisms $S_b$ to the categories appearing in this diagram:

\begin{defn}[\cf Theorem \protect{\ref{plim}}]\label{d4.1} Let
$X,Y\in\widehat{\Var}$, with $X=$ \allowbreak$\lim X_\alpha$, $Y=\lim
Y_\beta$. A morphism $s:X\to Y$ is \emph{birational} if, for each $\beta$,
the projection $X\by{s}Y\to Y_\beta$
factors through a birational map $s_{\alpha,\beta}:X_\alpha\to
Y_\beta$ for some $\alpha$ (this does not depend on the choice of
$\alpha$). We denote by $S_b\subset Ar(\widehat{\Var})$ the collection of
these morphisms.

In $\Var \P$, we write $S_b$ for the set of morphisms of the form $(u,f)$ where $u$ is an isomorphism of function fields and $f$ is a birational morphism. In $\place$, we take for $S_b$ the set of isomorphisms.
\end{defn}

\subsection{Main results}

\begin{defn}\label{d3.1} Let  
\begin{itemize}
\item $\place_*$ be the full subcategory of $\place$ formed of function fields which have a cofinal system of smooth proper models.  
\item $\Sm^\proper_*\subseteq \Sm^\proper$ be the full subcategory of those $X$ such that, for any $Y\in \Var^\proper$ birational to $X$, there exists $X'\in \Sm^\proper$ and a (proper) birational morphism $s:X'\to Y$.
\end{itemize}
\end{defn}

Note that $\Sm^\proper_*=\Sm^\proper$ in characteristic $0$ and that $X\in \Sm^\proper\allowbreak\Rightarrow X\in \Sm^\proper_*$ if $\dim X\le 2$ in any characteristic. On the other hand, it is not clear whether $\Sm_*^\proper$ is closed under products, or even under product with $\P^1$.


The following lemma is clear:

\begin{lemma}\label{l4.0}a)  If $X,X'\in \Sm^\proper$ are birational, then $X\in \Sm^\proper_*\allowbreak\iff X'\in \Sm^\proper_*$.\\
b) $K\in \place_*\iff$  $K$ has a model in $\Sm_*^\proper$, and then any smooth proper model of $K$ is in $\Sm_*^\proper$.\qed
\end{lemma}

If $X\in \Sm^\proper_*$, we have $F(X)\in\place_*$, hence 
with these definitions, \eqref{eq4.4a} induces a commutative diagram of localised categories:
\begin{equation}\label{eq4.4b}
\xymatrix{
&S_b^{-1}\Sm^\proper_* \P\ar[dr]^{\overline \Phi_2^*}\ar[dl]_{\overline \Phi_1^*}\\
\place_*^\op\ar[dr]^{\bar \Sigma} && S_b^{-1}\Sm^\proper_*\ar[dl]_{\bar J}\\
&S_b^{-1}\widehat{\Sm^\proper_*}
}
\end{equation}


\begin{thm}\label{t3.2} In \eqref{eq4.4b}, $\bar J$ and $\overline\Phi_1^*$ are equivalences of categories.
\end{thm}

Composing $\bar \Sigma$ with a quasi-inverse of $\bar J$, we get a functor
\begin{equation}\label{psism}
\Psi_*:\place_*^\op\to S_b^{-1}\Sm^\proper_*.
\end{equation}

This functor is well-defined up to unique natural isomorphism, by the essential uniqueness of a quasi-inverse to $\bar J$.

\begin{thm} \label{cplsm} a) The functor $\Psi_*$  is   full and essentially surjective.\\
 b) Let $K,L\in \place_*$ and $\lambda,\mu\in \place_*(K,L)$. Suppose that $\lambda$ and $\mu$ have the same centre on some model $X\in \Sm_*^\proper$ of $K$. Then $\Psi_*(\lambda)=\Psi_*(\mu)$. \\
c) Let  $S_r\subset \place_*$ denote the set of field extensions $K\inj K(t)$ such that $K\in \place_*$ and $K(t)\in \place_*$.
Then the composition $\place_*^\op\by{\Psi_*} S_b^{-1}\Sm^\proper_*\to S_b^{-1}\Sm$ factors through a (full) functor, still denoted by $\Psi_*$:
\[
\Psi_*:S_r^{-1}\place_*^\op\to S_b^{-1}\Sm.
\]
\end{thm}

The proofs of Theorems \ref{t3.2} and \ref{cplsm} go in several steps, which are given in the next subsections.

\subsection{Proof of Theorem \ref{t3.2}: the case of $\bar J$}


We apply Proposition 5.10 b) of
\cite{loc}. To lighten
notation we drop the functor $J$. We have to check
Conditions (b1), (b2) and (b3) of \loccit, namely:
\begin{itemize}
\item[(b1)] \emph{Given two maps $X\begin{smallmatrix}f\\ \rr\\ 
g\end{smallmatrix} Y$
in $\Sm^\proper_*$ and a map $s:Z=\lim Z_\alpha\to X$ in $S_b\subset
\widehat{\Sm^\proper_*}$, $fs=gs\Rightarrow f=g$.} This is clear by Lemma
\ref{l1}, since by
Theorem \ref{plim} $s$ factors through some $Z_\alpha$, with
$Z_\alpha\to X$ 
birational.
\item[(b2)] \emph{For any $X=\lim X_\alpha\in \widehat{\Sm^\proper_*}$, there
exists a birational morphism $s:X\to X'$ with $X'\in \Sm^\proper_*$.} It
suffices to take $X'=X_\alpha$ for some $\alpha$. 
\item[(b3)] \emph{Given a diagram
\[\begin{CD}
X_1\\
@A{s_1}AA\\
X=\lim X_\alpha @>f>> Y
\end{CD}\]
with $X\in \widehat{\Sm^\proper_*}$, $X_1,Y\in \Sm^\proper_*$ and $s_1\in S_b$,
there exists $s_2:X\to X_2$ in $S_b$, with $X_2\in \Sm^\proper_*$, covering
both $s_1$ and $f$.} Again, it suffices to take $X_2=X_\alpha$ for $\alpha$
large enough (use Theorem \ref{plim}).
\end{itemize} 

\subsection{Calculus of fractions}

\begin{prop}\label{l8a} The category $\Sm^\proper_*\P$ admits a calculus  of right fractions with respect to $S_b^p$. In particular, in $(S_b^p)^{-1}\Sm^\proper_*\P$, any  morphism may be written in the form $fp^{-1}$ with $p\in S_b^p$. The latter also holds in $(S_b^p)^{-1}\Sm^\proper_*$.
\end{prop}

\begin{proof} 
Consider a diagram 
\begin{equation}\label{eq3.4}
\begin{CD}
&& Y'\\
&&@V{s}VV\\
X@>u>>Y
\end{CD}
\end{equation}
 in $\Sm^\proper_*\P$, with $s\in
S_b^p$. Let
$\lambda:F(Y)\tto F(X)$ be the place compatible with $u$ which is implicit in
the statement. By Proposition \ref{l2}, $\lambda$ has centre $z=u(\eta_X)$
on $Y$. Since $s$ is proper, $\lambda$ therefore has also a centre $z'$ on
$Y'$. By Lemma \ref{l4} a), $s(z')=z$. By Lemma \ref{l2bis}, there
exists a unique  rational map $\phi:X\ttto Y'$ compatible with 
$\lambda$, and $s\circ \phi=u$ by Lemma \ref{l4} b). By the graph
trick, we 
get a commutative diagram
\begin{equation}\label{eq3.5}
\begin{CD}
X'@> u'>>  Y'\\
@V{s'}VV@V{s}VV\\
X@> u>>Y
\end{CD}
\end{equation}
in which $X'\subset X\times_Y Y'$ is the closure of the graph of
$\phi$, $s'\in S_b^p$ and $u'$ is compatible with $\lambda$. 
Since $X\in \Sm_*^\proper$, we may birationally dominate $X'$ by an $X''\in \Sm_*^\proper$ by Lemma \ref{l4.0}, hence replace $X'$ by $X''$ in the diagram.

Since $\Phi_1^*$ is full by Lemma \ref{l6}, the same construction works in $\Sm_*^\proper$, hence the structure of morphisms in $(S_b^p)^{-1}\Sm^\proper_*\P$ and $(S_b^p)^{-1}\Sm^\proper_*$.

Let now 
\[X\begin{smallmatrix}f\\ \rr\\ g
\end{smallmatrix} Y\overset{s}{\To} Y'
\]
be a diagram in $\Sm^\proper_*\P$ with $s\in S_b^p$, such that $sf=sg$. By
Corollary \ref{c1.1} c), the place underlying $s$ is the identity. Hence the
two places underlying $f$ and $g$ must be equal. But then $f=g$ by
Proposition \ref{l2}.
\end{proof}


\begin{prop}\label{p4a} a) Consider a diagram in $\Sm^\proper_*\P$ 
\begin{equation}\label{eq3.2}
\xymatrix{ 
&Z\ar[dl]_p\ar[dr]^f\\ 
X&& Y\\ 
&Z'\ar[ul]^{p'}\ar[ur]_{f'} 
}
\end{equation} 
where $p,p'\in S_b^p$. Let $K=F(Z)=F(Z')=F(X)$,  $L=F(Y)$
 and suppose given a place $\lambda:L\tto K$ compatible both with $f$
and $f'$. Then $(\lambda,fp^{-1}) =(\lambda,f'{p'}^{-1})$ in
$(S_b^p)^{-1}\Sm^\proper_*\P$. \\
b) Consider a diagram \eqref{eq3.2} in $\Sm^\proper_*$. Then $fp^{-1}=f'{p'}^{-1}$ in $(S_b^p)^{-1}\Sm^\proper_*$ if $(f,p)$ and $(f',p')$ define the same rational map from $X$ to $Y$.\\
\end{prop}

\begin{proof} 
a) By the graph trick, complete the diagram as
follows:
\begin{equation}\label{eq3.3}\xymatrix{
&Z\ar[dl]_p\ar[dr]^f\\ 
X&Z''\ar[u]^{p_1}\ar[d]_{p'_1}& Y\\ 
&Z'\ar[ul]^{p'}\ar[ur]_{f'} 
}\end{equation}
with $p_1,p'_1\in S_b^p$ and $Z''\in \Sm^\proper_*\P$. Since $X\in \Sm_*^\proper$, we may take $Z''$ in $\Sm_*^\proper$.Then we have
\[pp_1=p'p'_1,\quad fp_1=f'p'_1\]
(the latter by Lemma \ref{l4} b)), hence the claim. 

b) If $(f,p)$ and $(f',p')$ define the same rational map, then arguing as in a) we get a diagram \eqref{eq3.3} in $\Sm_*^\proper$, hence $fp^{-1}=f'{p'}^{-1}$  in $(S_b^p)^{-1}\Sm^\proper_*$. 
\end{proof}

\subsection{The morphism associated to a rational map}\label{mor-rat} Let $X,Y\in \Sm_*^\proper$, and let $\phi:Y\ttto X$ be a rational map. By the graph trick, we may find $p:Y'\to Y$ proper birational and a morphism $f:Y'\to X$ such that $\phi$ is represented by $(f,p)$; since $Y\in \Sm_*^\proper $, we may choose $Y'$ in $\Sm_*^\proper$. Then $fp^{-1}\in (S_b^p)^{-1}\Sm_*^\proper$ does not depend on the choice of $Y'$ by Proposition \ref{p4a} b): we simply write it $\phi$.

\subsection{Proof of Theorem \ref{cplsm}} \label{scplsm}

Let $K,L\in \place_*$ and $\lambda\in \place_*(K,L)$. Put  $X=\Psi_*(K), Y= \Psi_*(L)$, so that $X$ (resp. $Y$) is a smooth proper model of  $K$ (resp. $L$) in $\Sm_*$ (see \ref{d3.1}). Since $X$ is proper, $\lambda$ is finite on $X$ and by Lemma \ref{l2bis} there exists a unique  rational map $\phi:Y\ttto X$ compatible with $\lambda$, that we view as a morphism in $(S_b^p)^{-1}\Sm_*^\proper$ by \S \ref{mor-rat}.

\begin{lemma}\label{l3.3} With the above notation, we have $\Psi_*(\lambda)=\phi$.
\end{lemma}

\begin{proof}  Consider the morphisms $(\lambda,f)\in \Sm^\proper_*\P(Y',X)$ and $(1_L,s)\in \Sm^\proper_*\P(Y',Y)$. In \eqref{eq4.4a} $\Phi_1^*$ sends the first morphism to $\lambda$ and the second one to $1_L$, while $\Phi_2^*$ sends the first morphism to $f$ and the second one to $s$. The conclusion now follows from the commutativity of \eqref{eq4.4a} and the construction of $\Psi_*$.
\end{proof}

We can now prove Theorem \ref{cplsm}:

a) The essential surjectivity of $\Psi_*$ is tautological. Let now  $X=\Psi_*(K), Y= \Psi_*(L)$ for some $K,L\in \place_*$ and let $\phi\in (S_b^p)^{-1} \Sm^\proper_*(X,Y)$.  By Proposition \ref{l8a}, we may write $\phi=fs^{-1}$
where $f,s$ are morphisms in $\Sm^\proper_*$ and $s\in S_b^p$. Let $\tilde \phi:X\ttto Y$ be the corresponding rational map.
By Lemma \ref{l6}, $f$ is compatible with some place $\lambda$ and by Corollary \ref{c1.1} c),  $s$ is compatible with the corresponding isomorphism $\iota$ of function fields. Then $\tilde\phi$ is compatible with $\iota^{-1}\lambda$, and $\Psi_*(\iota^{-1}\lambda)=\phi$ by Lemma \ref{l3.3}. This proves the fullness of $\Psi_*$. (One could also use Lemma \ref{lfess}.)

 b) By Lemma \ref{l3.3}, $\Psi_*(\lambda)$ and $\Psi_*(\mu)$ are given by the respective rational maps $f,g:\Psi_*(L)\ttto \Psi_*(K)$ compatible with $\lambda,\mu$. By the definition of $\Sm_*^\proper$, we can find a model $X'\in \Sm_*^\proper$ of $K$ and two birational morphisms $s:X'\to X$, $t:X'\to \Psi_*(K)$. The hypothesis and Lemma \ref{l2bis} imply that $st^{-1}f=st^{-1}g$, hence $f=g$.

c) The said composition sends morphisms in $S_r$ to morphisms in $S_r$, hence induces a functor
\[
S_r^{-1}\place_*^\op\to S_r^{-1}\Sm.
\]

But $S_b^{-1}\Sm\iso S_r^{-1}\Sm$ by Theorem \ref{sb=sr}.



\subsection{Proof of Theorem \ref{t3.2}: the case of $\overline \Phi_1^*$}\label{s3.2}

Essential surjectivity is obvious by definition of $\place_*$. Let $X,Y\in \Sm_*^\proper\P$, and $K=\Phi_1^*(X), L=\Phi_1^*(Y)$.  By Lemma \ref{l2bis}, a place $\lambda:L\tto K$ is compatible with a (unique) rational map $\phi:X\ttto Y$. Since $X\in \Sm_*^\proper$, we may write $\phi=fs^{-1}$ with $f:X'\to Y$ for $X'\in \Sm_*^\proper$, and $s:X'\to X$ is a birational morphism. This shows the fullness of $\overline \Phi_1^*$. 

We now prove the faithfulness of $\overline \Phi_1^*$. Let $(\lambda_1,\psi_1),(\lambda_2,\psi_2)$ be two morphisms  from $X$ to $Y$ in $(S_b^p)^{-1}\Sm^\proper_*\P$ having the same image under $\overline \Phi_1^*$. By Proposition \ref{l8a}, we may write $\psi_i =f_ip_i^{-1}$ with $f_i,p_i$ morphisms and $p_i\in S_b$. As they have the same image, it  means that the places $\lambda_1$ and $\lambda_2$ from $F(Y)$ to $F(X)$ are equal. By Lemma \ref{l2bis}, $(f_1,p_1)$ and $(f_2,p_2)$ define the same rational map $\phi:X\ttto Y$. Therefore $\psi_1=\psi_2$ by Proposition \ref{p4a} b), and $(\lambda_1,\psi_1)=(\lambda_2,\psi_2)$.

\subsection{Dominant rational maps}\label{3.1} 
Recall from Subsection \ref{ratmaps} the category $\Rat_\dom$ of dominant rational maps between $F$-varieties. Writing $\Var_\dom$ for the category of $F$-varieties and dominant maps, we have inclusions of categories
\begin{equation}\label{hart2} \Var \supset \Var_\dom\overset{\rho}{\Inj} \Rat_\dom.
\end{equation}

Recall \cite[Ch. I, Th. 4.4]{hartshorne2} that there is an anti-equivalence of
categories
\begin{align}
\Rat_\dom&\iso \field^\op\label{hart}\\
X&\mapsto F(X).\notag
\end{align}

Actually this follows easily from Lemma \ref{l2.1}. We want to revisit
this theorem from the current point of view. For simplicity, we restrict to smooth varieties and separably generated extensions of $F$. Recall:

\begin{lemma}\label{lsg} A function field $K/F$ has a smooth model if and only if it is separably generated.
\end{lemma}

\begin{proof}  \emph{Necessity}: let $p$ be the exponential characteristic of $F$. If $X$ is a smooth model of $K/F$, then $X\otimes_F F^{1/p}$ is smooth over $F^{1/p}$ and irreducible, hence $K\otimes_F F^{1/p}$ is still a field. The conclusion then follows from Mac Lane's separability criterion \cite[Chapter 8, \S 4]{l}

\emph{Sufficiency}: if $K/F$ is separably generated, pick a separable transcendence basis $\{x_1,\dots,x_n\}$. Writing $F(x_1,\dots,x_n)=F(\A^n)$, we can find an affine model of finite type $X$ of $K/F$ with a dominant generically finite morphism $f:X\to \A^n$. By generic flatness \cite[11.1.1]{ega4}, there is an open subset $U\subseteq \A^n$ such that $f^{-1}(U)\to U$ is flat. On the other hand, since $K/F(x_1,\dots,x_n)$ is separable, there is another open subset $V\subseteq \A^n$ such that $\Omega^1_{f^{-1}(V)/V}=0$. Then $f^{-1}(U\cap V)$ is flat and unramified, hence \'etale, over $U\cap V$, hence is smooth over $F$ since $U\cap V$ is smooth \cite[17.3.3]{ega4}.
\end{proof}

 Instead of \eqref{eq4.4a} and \eqref{eq4.4b}, consider now the commutative diagrams of functors

\begin{equation}\label{eq4.4c}
\xymatrix{
&\Sm_\dom \P\ar[dr]^{\Phi_{2,\dom}}\ar[dl]_{\Phi_{1,\dom}}\\
\field_s^\op\ar[dr]^\Spec && \Sm_\dom\ar[dl]_{J_\dom}\\
&\widehat{\Sm_\dom}
}\; 
\xymatrix{
&S_b^{-1}\Sm_\dom \P\ar[dr]^{\overline \Phi_{2,\dom}}\ar[dl]_{\overline \Phi_{1,\dom}}\\
\field_s^\op\ar[dr]^{\overline{\Spec}} && S_b^{-1}\Sm_\dom\ar[dl]_{\bar J_\dom}.\\
&S_b^{-1}\widehat{\Sm_\dom}
}
\end{equation}

Here, $\field_s\subseteq \field$ is the full subcategory of separably generated extensions, $\Sm_\dom \P$ is the  subcategory of $\VarP$ given by varieties in $\Sm$ and morphisms are pairs $(\lambda,f)$ where $f$ is dominant (so that $\lambda$ is an inclusion of function fields) and $\Phi_{1,\dom}$, $\Phi_{2,\dom}$ are the two forgetful functors of \eqref{l3}, restricted to $\Sm_\dom\P$. Similarly, $J_\dom$ is the analogue of $J$ for $\Sm_\dom$. We extend the birational morphisms $S_b$ as in Definition \ref{d4.1}.

\begin{thm}\label{p3.2a} In the left diagram of \eqref{eq4.4c}, $\Phi_{2,\dom}$ is an isomorphism of categories. In the right diagram, all functors are equivalences of categories.
\end{thm} 

\begin{proof} The first claim follows from  Corollary \ref{c1.1} c). In the right diagram,  the proofs for $\bar J_\dom$ and $\overline{\Phi}_{1,\dom}$ are exactly parallel to those of Theorems \ref{t3.2} and  \ref{cplsm} with a much simpler proof for the latter.  As $\overline{\Phi}_{2,\dom}$ is an isomorphism of categories, the 4th functor $\overline{\Spec}$ is an equivalence of categories as well.
\end{proof}

In Theorem \ref{p3.2a}, we could replace $\Sm_\dom$ by $\Var_\dom$ or $\Var_\dom^\proj$ (projective varieties) and $\field_s$ by $\field$ (same proofs).\footnote{We could also replace dominant morphisms by flat morphisms, as in \cite{BG1}.} Since $\Phi_{2,\dom}$ is an isomorphism of categories in both cases, we directly get a naturally commutative diagram of categories and functors
\begin{equation}\begin{CD}
&&S_b^{-1}\Sm_\dom@>\sim>> \field_s^\op\label{eq4.3}\\
&&@VVV @VVV\\
 S_b^{-1}\Var_\dom^\proj@>\sim>> S_b^{-1}\Var_\dom@>\sim>> \field^\op.
\end{CD}\end{equation}
where the horizontal ones are equivalences.

To make the link with  \eqref{hart}, note that  the functor $\rho$ of \eqref{hart2} sends a 
birational morphism to an isomorphism. Hence $\rho$ induces functors
\begin{equation}\label{eq3.0}
S_b^{-1}\Var^\proj_\dom\to S_b^{-1}\Var_\dom\to \Rat_\dom
\end{equation}
whose composition with the second equivalence of \eqref{eq4.3} is \eqref{hart}.

\begin{prop}\label{p3.2} Let 
$S=S_o,S_b$ or $S_b^p$. \\
a) $S$ admits
a calculus of right fractions within $\Var_\dom$.\\
b) The functors in \eqref{eq3.0} are equivalences of categories.
\end{prop}

\begin{proof} a) For any pair
$(u,s)$ of morphisms as in Diagram \eqref{eq3.4},
with $s\in S$ and $u$ dominant, the pull-back of $s$ by $u$
exists
and is in $S$. Moreover, if $sf=sg$ with $f$ and $g$
dominant and $s\in S$, then $f=g$. 

b) This follows from \eqref{eq4.3} and \eqref{hart}.
\end{proof}

Taking a quasi-inverse of \eqref{eq4.3}, we now get an equivalence of categories
\begin{equation}\label{eq4.3a}
\Psi_\dom:\field_s^\op\iso S_b^{-1}\Sm_\dom 
\end{equation}
which will be used in Section \ref{HR}.

\begin{rk} \label{r4.6} The functor $(S_b^p)^{-1}\Var_\dom\to
\field^\op$ is not full (hence is not an equivalence of
categories). For example, let $X$ be a proper variety and $Y$ an
affine open subset of $X$, and let $K$ be their common function
field. Then the identity map $K\to K$ is not in the image of the
above functor. Indeed, if it were, then by calculus of fractions it
would be represented by a map of the form $fs^{-1}$ where $s:X'\to
X$ is proper birational. But then $X'$ would be proper and $f:X'\to
Y$ should be constant, a contradiction. 

It can be shown that the
localisation functor 
\[(S_b^p)^{-1}\Var_\dom\to
S_b^{-1}\Var_\dom\] 
has a (fully faithful) right adjoint given by
\[(S_b^p)^{-1}\Var^\proj_\dom\to
(S_b^p)^{-1}\Var_\dom\] 
via the equivalence
$(S_b^p)^{-1}\Var^\proj_\dom\iso
S_b^{-1}\Var_\dom$ given by Proposition \ref{p3.2} b). The
proof is similar to that of Theorem \ref{t3.1} (ii) below.
\end{rk}

\subsection{Recapitulation}\label{recap} We 
constructed a full and essentially surjective functor (Theorem \ref{cplsm})
\[
\Psi_*:S_r^{-1}\place_*^{\op} \to S_b^{-1}\Sm
\]
and an equivalence of categories \eqref{eq4.3a}
\[
\Psi_\dom=\bar J_\dom^{-1} \circ \overline{\Spec}:\field_s^\op\iso S_b^{-1}\Sm_\dom. 
\]

 Consider the natural functor 
\begin{equation}\label{eq.natural}
\theta:S_b^{-1}\Sm_*^\proj \to S_b^{-1}\Sm.
\end{equation}

In characteristic zero, $\theta$ is an equivalence of categories by \cite[Prop. 8.5]{loc}, noting that in this case $\Sm_*^\proj= \Sm^\proj$ by Hironaka. Let $\iota$ be the inclusion $\field_s^\op \hookrightarrow \place_*^\op.$ Then the natural transformation $\eta:\Spec \Rightarrow\Sigma$ of \eqref{eqeta} provides the following naturally commutative diagram
\begin{equation}\label{jolidiagramme}
\xymatrix{
\field_s^\op\ar[r]_\sim^{\Psi_\dom}\ar[d]^\iota&S_b^{-1}\Sm_\dom\ar[r] &S_b^{-1}\Sm\\
S_r^{-1}\place_*^{\op}\ar[r]^{\Psi_*}& S_b^{-1}\Sm_*^\proj.\ar[ur]_\sim^{\theta}
}
\end{equation}

(Note that $\eta$  induces a natural isomorphism $\bar \eta:\overline{\Spec} \overset{\sim}{\Rightarrow} \bar \Sigma$.)

In characteristic $p$, we don't know if $\field_s\subset \place_*$: to get an analogue of \eqref{jolidiagramme} we would have to take the intersection of these  categories.   We shall do this   in Section \ref{HR} in an enhanced way, using a new idea (Lemma \ref{l5.4} a)). As a byproduct, we shall get the full faithfulness of $\theta$ in any characteristic (Corollary \ref{c5.1})

\section{Other classes of varieties}\label{eqcat}

In this section we prove that, given a full subcategory $\sC$ of $\Var$
satisfying certain hypotheses, the functor
\[S_b^{-1}\sC\P\to \place^\op\]
induced by the functor $\Phi_1$ of Diagram \eqref{eq4.4a} is fully faithful.


\subsection{The $_*$ construction} We generalise Definition \ref{d3.1} as follows:

\begin{defn}\label{d5.1} Let $\sC$ be a full subcategory of $\Var$. We write $\sC_*$ for the full subcategory of $\sC$ with the following objects: $X\in \sC_*$ if and only if, for any $Y\in \Var^\proper$ birational to $X$, there exists $X'\in \sC$ and a proper birational morphism $s:X'\to Y$.
\end{defn}

\begin{lemma}\label{l5.1} a) $\sC_*$ is closed under birational equivalence.\\
b) We have $\sC_*=\sC$ for the following categories: $\Var, \Norm$ *and $\Sm,\Sm^\qp$ if $\car F=0$.\\
c) We have $\sC_*\cap \sC^\proper= (\sC^\proper)_*$, where $\sC^\proper:=\Var^\proper\cap \sC$.
\end{lemma}

\begin{proof} a) is tautological.  b) is trivial for $\Var$, is true for $\Norm$ because normalisation is finite and birational in $\Var$, and follows from Hironaka's resolution for $\Sm$.
Finally, c) is trivial.
\end{proof}

\begin{lemma}\label{l5.2} Suppose $\sC$ verifies the following condition: given a diagram
\[\begin{CD}
X'@>j>> \tilde X\\
@VpVV\\
X
\end{CD}\]
with $X,\tilde X\in \sC_*$, $p\in S_b^p$, $j\in S_o$ and $\tilde X$ proper, we have $X'\in \sC$. (This holds in the following special cases: $\sC\subseteq \Var^\proper$, or $\sC$  stable under open immersions.)\\ 
a) Let $X\in \sC_*$. Then the following holds: for any $s:Y\to X$ with $Y\in \Var$ and $s\in S_b^p$, there exists $t:X'\to Y$ with $X'\in \sC_*$ and $t\in S_b^p$.\\
b) Let $X,Y\in \sC_*$ with $Y$ proper, and let $\gamma:X\ttto Y$ be a rational map. Then there exists $X'\in \sC_*$, $s:X'\to X$ in $S_b^p$ and a morphism $f:X'\to Y$ such that $\gamma=fs^{-1}$.
\end{lemma}

\begin{proof} a) By Nagata's theorem, choose a compactification $\bar Y$ of $Y$. By hypothesis, there exists $\bar X'\in \sC$ and a proper birational morphism $t':\bar X'\to \bar Y$. If $X'={t'}^{-1}(Y)$, then $t:X'\to Y$ is a proper birational morphism. The hypothesis on $\sC$ then implies that $X'\in \sC$, hence $X'\in \sC_*$ by Lemma \ref{l5.1} a).

b) Apply a) to the graph of $\gamma$, which is proper over $X$.
\end{proof}

\subsection{Calculus of fractions}


\begin{prop}\label{p5.1} Under the condition of Lemma \ref{l5.2}, Propositions \ref{l8a} and \ref{p4a} remain valid for $\sC_*\P$. In particular, any morphism in $(S_b^p)^{-1}\sC_*\P$ or $(S_b^p)^{-1}\sC_*$ is of the form $fp^{-1}$, with $f\in \sC_*\P$ or $\sC_*$ and $p\in S_b^p$.
\end{prop}

\begin{proof} Indeed, the only fact that is used in the proofs of Propositions \ref{l8a} and \ref{p4a} is the conclusion of Lemma \ref{l5.2} a).
\end{proof}

To go further, we need:

\begin{prop}\label{l3.2}
In $(S_b^p)^{-1}\sC_*\P$, $S_o$ admits a calculus of left fractions.  In particular (\cf Proposition \ref{p5.1}),
any morphism in $S_b^{-1}\sC_*\P$ may be written as
$j^{-1}fq^{-1}$, with $j\in S_o$ and $q\in S_b^p$.
\end{prop}

\begin{proof} a) Consider a diagram in $(S_b^p)^{-1}\sC_*\P$
\[\xymatrix{
X\ar[r]^{j}\ar[d]_\phi& X'\\
Y
}\]
with $j\in S_o$. By
Proposition \ref{p5.1},
we may write $\phi=fp^{-1}$ with $p\in S_b^p$ and $f$ a morphism of $\sC_*\P$  ($f,p$ originate from some common $\bar X$). We may embed $Y$ as an open subset of a proper $\bar Y$. This gives us a rational map $X'\ttto \bar Y$. Using the graph trick, we may ``resolve" this rational map into a morphism $g:\tilde X'\to \bar Y$, with $\tilde X'\in \Var$ provided with a proper birational morphism $q:\tilde X'\to X'$.  Since $Y\in \sC_*$, we may assume $\bar X'\in \sC_*$. Let $\psi = gq^{-1}\in (S_b^p)^{-1}\sC_*\P$. Then the diagram in $(S_b^p)^{-1}\sC_*\P$
\[\xymatrix{
X\ar[r]^{j}\ar[d]_\phi& X'\ar[d]^\psi\\
Y\ar[r]^{j_1}& \bar Y
}\]
commutes because the following bigger diagram commutes in $\sC_*\P$:
\[\xymatrix{
\tilde X\ar[dr]^p\ar[ddr]_f &&\tilde X''\ar[ll]_{r}\ar[rr]^{r'}&& \tilde X'\ar[dl]_q\ar[ddl]^g\\
&X\ar[rr]^{j}&& X'\\
&Y\ar[rr]^{j_1}&& \bar Y
}\]
thanks to Lemma \ref{l4}, for suitable $\tilde X''\in \sC_*$ and $r,r'\in S_b^p$.

b) Consider a diagram
\[X'\overset{j}{\to}X\begin{smallmatrix}f\\ \rr\\ g
\end{smallmatrix} Y
\]
in $(S_b^p)^{-1}\sC_*\P$, where $j\in S_o$ and $fj = gj$. By Proposition \ref{p5.1}, we may write $f=\tilde f p^{-1}$ and $g=\tilde g p^{-1}$, where $\tilde f,\tilde g$ are morphisms in $\sC_*\P$ and $p:\tilde X\to X$ is in $S_b^p$. Let $U$ be a common open subset to $X'$ and $\tilde X$: then the equality $fj=gj$ implies that the restrictions of $\tilde f$ and $\tilde g$ to $U$ coincide as morphisms of $(S_b^p)^{-1}\sC_*\P$.  Hence the places underlying $\tilde f$ and $\tilde g$ are equal, which implies that $\tilde f=\tilde g$ (Proposition \ref{l2}), and thus $f=g$.
\end{proof}

\begin{rk} \label{r3.1}  $S_o$ does not
  admit a
  calculus of \emph{right} fractions, even in $(S_b^p)^{-1}\VarP$. Indeed, consider a diagram in
  $(S_b^p)^{-1}\VarP$
\[\xymatrix{
&Y'\ar[d]^j\\
X\ar[r]^f& Y
}\]
where $j\in S_o$ and, for simplicity, $f$ comes from
$\VarP$. Suppose that we can complete this diagram into a
commutative diagram in $(S_b^p)^{-1}\VarP$
\[\xymatrix{
\tilde X'\ar[d]_p\ar[dr]^g\\
X'\ar[r]^\phi\ar[d]_{j'}&Y'\ar[d]^j\\
X\ar[r]^f& Y
}\]
with $p\in S_b^p$ and $g$ comes from $\VarP$. By Proposition
\ref{l2} the localisation functor $\VarP\to (S_b^p)^{-1}\VarP$
is faithful, so the diagram (without $\phi$) must already commute in
$\VarP$. If $f(X)\cap Y'=\emptyset$, this is impossible.
\end{rk}

\subsection{Generalising Theorem \ref{t3.2}}


\begin{thm} \label{t3.1} Let $\sC$ be a full subcategory of $\Var$.  In diagram \eqref{eq4.4a},\\
a) $J$ induces an equivalence of categories $S_b^{-1}\sC\to S_b^{-1}\hat{\sC}$.\\
b) Suppose that $\sC$ verifies the condition of Lemma \ref{l5.2}. Consider the string of functors
\[(S_b^p)^{-1}\sC_*^\proper\P\by{S}
  (S_b^p)^{-1}\sC_*\P \by{T}
 S_b^{-1}\sC_*\P\by{\overline{\Phi_1^*}} \place^\op.\]
where $S$ and $T$ are the obvious ones and $\overline{\Phi_1^*}$ is induced by $\Phi_1$.
 Then
\begin{thlist}
\item $S$ is fully faithful and $T$ is faithful.
\item For any $X\in
  (S_b^p)^{-1}\sC_*P$ and $Y\in (S_b^p)^{-1}\sC_*^\proper\P$, the map 
\begin{equation}\label{eq3.1}T:\Hom(X,S(Y))\to \Hom(T(X), TS(Y))
\end{equation} 
is an isomorphism.
\item  $TS$ is an equivalence of categories.
\item  $\overline{\Phi_1^*}$ is fully faithful.
\end{thlist}
\end{thm}

\begin{proof} a) It is exactly the same proof as for the case of $\bar J$ in Theorem \ref{t3.2}.

b) In 4 steps:

A) We run through the proof of Theorem \ref{t3.2}  given in \S \ref{s3.2} for $\overline{\Phi_1^*}$ in the case $\sC=\Sm^\proj$. In view of Proposition \ref{p5.1},  the proof of faithfulness for $\overline{\Phi_1^*}T$  goes through verbatim. The proof of fullness for $\overline{\Phi_1^*}TS$ also goes through (note that in \loccit, we need $Y$ to be proper in order for $\lambda$ to be finite on it).
It follows that $S$ is fully faithful and $T$ is faithful. 

B) By A), \eqref{eq3.1} is injective. Let $\phi\in\Hom(T(X), TS(Y))$. By Proposition \ref{l3.2},
$\phi=j^{-1}fp^{-1}$ with $j\in S_o$
and $p\in S_b^p$. Since $Y$ is proper, $j$ is
necessarily an isomorphism, which shows the surjectivity of \eqref{eq3.1}. This proves (ii).

C) It follows from A) and B) that $TS$ is fully faithful. Essential
surjectivity follows from Lemma \ref{l5.1} a) and c) plus Nagata's theorem. This proves (iii).

D) We come to the proof of (iv). Since $\overline{\Phi_1^*}TS$ is faithful (see A)) and $TS$ is an
equivalence, $\overline{\Phi_1^*}$ is faithful. To show that it is 
full, let $X,Y\in \sC_*\P$ and
$\lambda:F(Y)\tto F(X)$ a place. Let $Y\to\bar Y$ be a
compactification of $Y$. By Definition \ref{d5.1}, we may choose $\bar Y'\by{s}\bar
Y$ with  $s\in S_b^p$ and $\bar Y'\in\sC_*^\proper$. 
Then $\lambda$ is finite over  $\bar Y'$. By Lemma \ref{l2bis}, there is a
rational map $f:X\ttto \bar Y'$ compatible with $\lambda$.  
Applying Lemma \ref{l5.2} b) to the rational maps $X\ttto \bar Y'$ and $Y\ttto \bar Y'$,  we find a diagram in $\sC_*$
\[\xymatrix{
X'\ar[d]^t \ar[r]^f & \bar Y' & Y'\ar[l]_{t'} \ar[d]^{s'}\\
X &  &Y
}\]
with $t,s'\in S_b^p$ (and $t'\in S_b$). Then $\phi=s'{t'}^{-1}ft^{-1}:X \to Y$ is such that
$\overline{\Phi_1^*}(\phi)=\lambda$.
\end{proof}



\begin{cor}\label{c5.3} The localisation functor $T$ has a right adjoint, given explicitly by $(TS)^{-1}\circ S$.\qed
\end{cor}

Consider now the commutative diagram of functors:
\begin{equation}\label{eq5.2}
\begin{CD}
(S_b^p)^{-1}\sC_*^\proper\P@>S>>
  (S_b^p)^{-1}\sC_*\P @>T>>
 S_b^{-1}\sC_*\P@>\overline{\Phi_1^*}>> \place^\op\\
 @VVV @VVV @VVV ||\\
(S_b^p)^{-1}\Var^\proper\P@>S>>
  (S_b^p)^{-1}\VarP @>T>>
 S_b^{-1}\VarP@>\overline{\Phi_1^*}>> \place^\op.
 \end{CD}
 \end{equation}

\begin{cor}\label{c5.2}  All vertical functors in \eqref{eq5.2}
are fully faithful.
\end{cor}


\begin{proof}  For the first and third vertical functors, this is a byproduct of Theorem
\ref{t3.1}. The middle one is faithful by the faithfulness of $T$
and $\overline{\Phi_1^*}$ in Theorem  \ref{t3.1}. For fullness, let $X, Y\in (S_b^p)^{-1}\sC_*\P$ and $\phi:X\to Y$ be a morphism in $(S_b^p)^{-1}\VarP$. By  Proposition \ref{p5.1}, we may write $\phi = fp^{-1}$, with $p:\tilde X\to X$
proper birational. By Lemma \ref{l5.2} a), we may find $p':\tilde X'\to \tilde X$
proper birational with $\tilde X'\in \sC_*$, and replace $fp^{-1}$ by $fp'(pp')^{-1}$. 
\end{proof}

\begin{rks} 1) Take $\sC=\Var$ in Theorem \ref{t3.1} and let $X,Y\in (S_b^p)^{-1}\VarP$. Then the image of $\Hom(X,Y)$ in  $\Hom(\overline{\Phi_1^*}T(Y),\overline{\Phi_1^*}T(X))$ via $\overline{\Phi_1^*}T$ is contained in the set of places which  are finite on $Y$.  If $X$ and $Y$ are
proper, then the image is all of $\Hom(\overline{\Phi_1^*}T(Y), \overline{\Phi_1^*}T(X))$. On the other
hand, if $X$ is proper and $Y$ is affine, then for any map
$\phi=fp^{-1}:X\to Y$, the source $X'$ of $p$ is proper hence $f(X')$
is a closed point of $Y$, so that the image is contained in
the set of places from $F(Y)$ to $F(X)$ whose centre on $Y$ is a
closed point (and one sees easily that this inclusion is an
equality). In general, the description of this image seems to depend heavily on the geometric 
nature of $X$ and $Y$.\\
2) For ``usual'' subcategories $\sC\subseteq \Var$, the functors $\Phi_1^*$, $\Phi_1^*T$ and $\Phi_1^*TS$  of Theorem \ref{t3.1} b) are essentially surjective (hence so are those in Corollary \ref{c5.2}): this is true for $\sC=\Var$ or $\Norm$ (any function field has a normal proper model), and for $\sC=\Sm$ in characteristic $0$. For $\sC=\Sm$ in positive characteristic, the essential image of these functors is the category $\place_*^\op$ of  Definition \ref{d3.1}.
\end{rks}

\subsection{Localising $\sC_*$} In Theorem \ref{t3.1}, we generalised Theorem \ref{t3.2} which was used to construct the functor $\Psi_*$ of \eqref{psism}. A striking upshot is Corollary \ref{c5.2}. What happens if we study $S_b^{-1}\sC_*$ instead of $S_b^{-1}\sC_*\P$?

This was done previously in \cite[\S 8]{loc}, by completely different methods. The two main points were:

\begin{itemize}
\item In characteristic $0$, we have the
following equivalences of categories:
\begin{equation}\label{eq5.3}
S_b^{-1}\Sm^\proj\simeq S_b^{-1}\Sm^\proper\simeq S_b^{-1}\Sm^\qp
\simeq S_b^{-1}\Sm
\end{equation}
induced by the obvious inclusion functors \cite[Prop. 8.5]{loc}.
\item Working with varieties that are not smooth or at least regular leads to pathologies: for example, the functor $S_b^{-1} \Sm\allowbreak \to S_b^{-1}\Var$ is neither full nor faithful \cite[Rk. 8.11]{loc}. This contrasts starkly with Corollary \ref{c5.2}. The issue is closely related to the regularity condition appearing in Lemma \ref{l6}; it is dodged in \cite[Prop. 8.6]{loc} by restricting to those morphisms that send smooth locus into smooth locus.
\end{itemize}

Using the methods of \cite{loc}, one can show that the functor
\begin{equation}\label{eq5.4}
(S_b^p)^{-1}\sC_*^\proper=S_b^{-1}\sC_*^\proper\to S_b^{-1} \sC_* 
\end{equation}
is an equivalence of categories for any $\sC\subseteq \Var$ satisfying the condition of Lemma \ref{l5.2}. For this, one should use \cite[Th. 5.14]{loc} under a form similar to that given in \cite[Prop. 5.10]{loc}. One can then deduce from Corollary \ref{c5.3} that the localisation functor
\[(S_b^p)^{-1} \sC_*\by{T} S_b^{-1} \sC_* \] 
has a right adjoint given (up to the equivalence \eqref{eq5.4}) by $(S_b^p)^{-1} \sC_*^\proper\by{S} (S_b^p)^{-1} \sC_*$ (in particular, $S$ is fully faithful): indeed, the unit and counit of the adjunction in Corollary \ref{c5.3} map by the essentially surjective forgetful functors 
\begin{equation}\label{eq5.5}
S_b^{-1}\sC_*\P\to S_b^{-1}\sC_*, \quad \text{etc.}
\end{equation}
 to natural transformations which keep enjoying the identities of an adjunction. Note however that \eqref{eq5.5} is not full unless $\sC\subseteq \Sm$ (see Lemma \ref{lfess} and Lemma \ref{l2} for this case).

For $\sC=\Sm$ or $\Sm^\qp$, the equivalence \eqref{eq5.4} extends a version of \eqref{eq5.3} to positive characteristic. We won't give a detailed proof however, because it would be tedious and we shall obtain a better result later (Corollary \ref{c5.1}) by a different method.

The proofs given in \cite{loc} do not use any calculus of fractions. In fact,  $S_b^p$
\emph{does not admit any calculus of fractions within $\Var$}, contrary to the case
of $\VarP$ (\cf Proposition \ref{l8a}). This is shown by the same 
examples as in Remark \ref{rk2.1}. If we restrict to $\Sm_*$,
we can use Proposition \ref{l8a} and Lemma \ref{l6} to prove a helpful part of calculus of fractions:

\begin{prop}\label{p3.4} a)   Let $s:Y\to X$ be
  in $S_b^p$, with $X$ smooth. Then $s$ is an envelope \cite{fulgil}: for any
  extension  $K/F$, the map $Y(K)\to X(K)$ is surjective.\\
b) The  multiplicative set $S_b^p$ verifies the second axiom of
  calculus of right fractions within $\Sm_*$.\\
c) Any morphism in $S_b^{-1}\Sm_*$ may be represented as $j^{-1}fp^{-1}$,
where $j\in S_o$ and $p\in S_b^p$.
\end{prop}

\begin{proof} a) Base-changing to $K$, it suffices to deal with $K=F$. 
Let $x\in X(F)$. By lemma \ref{l6}, there is a place
$\lambda$ of $F(X)$ with centre $x$ and residue field $F$. The
valuative criterion for properness implies that $\lambda$ has a centre
$y$ on $Y$; then $s(y)=x$ by Lemma \ref{l4} and $F(y)\subseteq
F(\lambda)=F$. 

b) We consider a diagram \eqref{eq3.4} in $\Sm_*$, with $s\in
S_b^p$. By a), $z=u(\eta_X)$ has a preimage $z'\in Y'$
with same residue field. Let $Z=\overline{\{z\}}$ and
$Z'=\overline{\{z'\}}$: the map $Z'\to Z$ is birational. Since the map
$\bar u:X\to Z$ factoring $u$ is dominant, we get by Theorem \ref{p3.2}
b) a commutative diagram like
\eqref{eq3.5}, with $s'$ proper birational. By Lemma \ref{l5.2} a), we may then replace $X'$ by an object of $\Sm_*$.

c) As that of Proposition \ref{l3.2}.
\end{proof}

\begin{rk} On the other hand, $S_b^p$ is far from
verifying the third axiom of calculus of right fractions within $\Sm_*$. Indeed, let
  $s:Y\to X$ be a proper birational morphism that contracts some
  closed subvariety $i:Z\subset Y$ to a point. Then, given any two morphisms
  $f,g:Y'\rr Z$, we have $sif=sig$. But if $ift=igt$ for some $t\in
  S_b^p$, then $if=ig$ (hence $f=g$) since $t$ is dominant.
 \end{rk}

\section{Homotopy of places and $R$-equivalence}\label{HR}

In this section, we do several things. In Subsection \ref{s5.1} we prove elementary results on divisorial valuations with separably generated residue fields. In Subsection \ref{s.dv} we introduce a subcategory $\dv$ of $\place$, where morphisms are generated by field inclusions and places given by discrete valuation rings. We relate it in Subsection \ref{s.am} with a construction of Asok-Morel \cite{a-m} to define a functor
\[\Psi:S_r^{-1}\dv\to S_b^{-1}\Sm\]
extending the functor $\Psi_\dom$ of \eqref{eq4.3a}. This functor is compatible with the functor $\Psi_*$ of Theorem \ref{cplsm}.  We then show in Proposition \ref{leq1} that the localisation $\place\to S_r^{-1}\place$ is also a quotient by a certain equivalence relation $\h$; although  remarkable, this fact is elementary.

Next, we reformulate a result of Asok-Morel to enlarge the equivalence relation $\h$ to another, $\h'$, so that the functor $\Psi$ factors through an equivalence of categories
\[\dv/\h'\iso S_b^{-1}\Sm.\]

Finally, we use another result of Asok-Morel to compute some Hom sets in $S_b^{-1}\Sm$ as $R$-equivalence classes: in the first version of this paper, we had proven this only in characteristic $0$ by much more complicated arguments.

\subsection{Good dvr's}\label{s5.1}

\begin{defn}\label{d-good} A discrete valuation ring (dvr) $R$ containing $F$ is \emph{good} if its quotient field $K$ and its residue field $E$ are finitely and separably generated over $F$, with $\trdeg(E/F)=\trdeg(K/F)-1$.
\end{defn}

\begin{lemma}\label{l4.3} A dvr $R$ containing $F$ is good if and only if there exist a smooth $F$-variety $X$ and a smooth divisor $D\subset X$ such that $R\simeq \sO_{X,D}$.
\end{lemma}

\begin{proof} Sufficiency is clear by Lemma \ref{lsg}.  Let us show necessity. The condition on the transcendence degrees means that $R$ is \emph{divisorial} = a ``prime divisor'' in the terminology of \cite{zs}. By \loccit, Ch. VI, Th. 31, there exists then a model $X$ of $K/F$ such that $R=\sO_{X,x}$ for some point $x$ of codimension $1$. (In particular, granting the finite generation of $K$, that of $E$ is automatic.) Furthermore, the separable generation of $E$ yields a short exact sequence
\[0\to \fm/\fm^2\to \Omega^1_{R/F}\otimes_R E\to \Omega^1_{E/F}\to 0\]
where $\fm$ is the maximal ideal of $R$ (see Exercise 8.1 (a) of \cite[Ch. II]{hartshorne2}). Therefore $\dim_E \Omega^1_{R/F}\otimes_R E=\trdeg(K/F) =\dim_K \Omega^1_{R/F}\otimes_R K$, thus $\Omega^1_{R/F}$ is free of rank $\trdeg(K/F)$ and $x$ is a smooth point of $X$. Shrinking $X$  around $x$, we may assume that it is smooth; if $D=\overline{\{x\}}$,  it is generically smooth by Lemma \ref{lsg}, hence we may assume $D$ is smooth up to shrinking $X$ further.
\end{proof}

\begin{lemma}\label{lsg2} Let $R$ be a good dvr containing $F$, with quotient field $K$ and residue field $E$, and let $K_0/F$ be a subextension of $K/F$. Then $R\cap K_0$ is either $K_0$ or a good dvr.
\end{lemma}

\begin{proof} By Mac Lane's criterion, $K_0$ is separably generated, and the same applies to the residue field $E_0\subseteq E$ of $R\cap K_0$ if the latter is a dvr.
\end{proof}

\subsection{The category $\dv$}\label{s.dv}

\begin{defn} Let  $K/F$ and $L/F$ be two separably generated extensions. We denote by $\dv(K,L)$ the set of morphisms in $\place(K,L)$ of the form
\begin{equation}\label{eq-dv}
K\tto K_1\tto\dots \tto K_n\Inj L  
\end{equation}
where for each $i$, the place $K_i\tto K_{i+1}$ corresponds to a good dvr with quotient field $K_i$ and residue field $K_{i+1}$. (Compare \cite[Ch. VI, \S 3]{zs}.)
\end{defn}

\begin{lemma}\label{l4.1} In $\dv(K,L)$, the decomposition of a morphism in the form \eqref{eq-dv} is unique. The collection of the $\dv(K,L)$'s defines a subcategory $\dv\subset \place$, with objects the separably generated function fields.
\end{lemma}

\begin{proof} Uniqueness follows from \cite[p. 10]{zs}. To show that $Ar(\dv)$ is closed under composition, we immediately reduce to the case of a composition
\begin{equation}\label{eq.elem}
K\overset{i}{\Inj} L\overset{\lambda}{\tto} L_1
\end{equation}
where $(L,L_1)$ correspond to a good dvr $R$. Then the claim follows from applying Lemma \ref{lsg2} to the commutative diagram in $\place$
\begin{equation}\label{eq4.1}
\begin{CD}
L@>\lambda>> L_1\\
@AiAA @Ai_1AA\\
K@>\lambda_1>> K_1
\end{CD}
\end{equation}
where $K_1$ is the residue field of $R\cap K$ if this is a dvr, and $K_1=K$ otherwise (and then $\lambda_1$ is a trivial place).
\end{proof}

We shall need the following variant of a theorem of Knaf and Kuhl\-mann \cite[Th. 1.1]{knaf-kuhlmann} (compare \loccit, pp. 834/835):

\begin{thm}\label{t0} Let $\lambda:K\tto L$ be a morphism in $\dv$. Then $\lambda$ is finite over a smooth model of $K$. Moreover, let $K'\subseteq K$ be a subextension of $K$, and let $Z$ be a model of $K'$ on which $\lambda_{|K'}$ has a centre $z$. Then there is a smooth model $X$ of $K$ on which $\lambda$ has a centre of codimension $n$, the rank of $\lambda$, and a morphism $X\to Z$ inducing the extension $K/K'$.
\end{thm}

\begin{proof} This actually follows from \cite[Th. 1.1]{knaf-kuhlmann}\footnote{We thank Hagen Knaf for his help in this proof.}: let $U$ be an
open affine neighbourhood of $z$ and let $E:=\{y_1,\dots,y_r\}$ be a set of generators of the
$F$-algebra $\sO_{Z}(U)$ (ring of sections). Then by \cite[Th. 1.1]{knaf-kuhlmann}, there exists a model $X_0$ of $K/F$ such that:
\begin{itemize}
\item $\lambda$ is centred at a smooth point $x$ of $X_0$,
\item $\dim\sO_{X_0,x}=n=\dim \sO_\lambda$,
\item $E$ is contained in the maximal ideal of $\sO_{X_0,x}$.
\end{itemize}

Hence  $\sO_{Z}(U)\subseteq \sO_{X_0}(X)$ for some open affine neighbourhood $X$ of $x$,
which yields a morphism $X\to U$ that maps $x$ to $z$.
\end{proof}

\subsection{Relationship with the work of Asok and Morel}\label{s.am}

In \cite[\S 6]{a-m}, Asok and Morel prove closely related results: let us translate them in the present setting. 

Let us write $\sC^\vee$ for the category of presheaves of sets on a category $\sC$. In \cite{a-m}, the authors denote the category $(S_r^{-1}\Sm)^\vee$ by $Shv_F^{h\A^1}$. Similarly, they write $\sF_F^r-\Set$ for  the  category consisting of objects of $(\field_s^\op)^\vee$ provided with ``specialisation maps'' for good dvrs. In \cite[Th. 6.1.7]{a-m},  they construct a full embedding 
\begin{equation}\label{eq.am}
Shv_F^{h\A^1}\to \sF_F^r-\Set
\end{equation}
 (evaluate presheaves on function fields),  and show that its essential image consists of those functors $\sS\in \sF_F^r-\Set$ satisfying a list of axioms (A1) -- (A4) (ibid., Defn. 6.1.6). 

The proof of Lemma \ref{l4.1} above shows that Conditions (A1) and (A2) mean that $\sS$ defines a functor $\dv^\op\to \Set$, and Condition (A4) means that $\sS$ factors through $S_r^{-1}\dv^\op$. 
 In other words, they essentially\footnote{Essentially because Condition (A1) of \cite[\S 6]{a-m} only requires a commutation of diagrams coming from \eqref{eq4.1} when the ramification index is $1$.} construct a functor
\[(S_r^{-1}\Sm)^{\vee}\to (S_r^{-1}\dv^\op)^\vee.\]
 
 We now check that this functor is induced by a functor
\begin{equation}\label{eqPsi}\Psi: S_r^{-1}\dv^\op\to S_b^{-1}\Sm.
\end{equation}

For this, we need a lemma:

\begin{lemma}\label{l6.1} Let $\Sm^\ess$ be the category of irreducible smooth $F$-sche\-mes essentially of finite type. Then the full embedding $\Sm\inj \Sm^\ess$ induces an equivalence of categories
\[S_b^{-1}\Sm\iso S_b^{-1}\Sm^\ess.\] 
\end{lemma}

\begin{proof} We use again the techniques of \cite{loc}, to which we refer the reader: actually the first part of the proof of \cite[Prop. 8.4]{loc} works with a minimal change. Namely, with notation as in \loccit, there are 3 conditions (b1) -- (b3) to check:
\begin{itemize}
\item[(b1)] \emph{Given $f,g:X\to Y$ in $\Sm$ and $s:Z\to X$ in $\Sm^\ess$ with $s\in S_b$, $fs=gs$ $\Rightarrow$ $f=g$}: this follows from Lemma \ref{l1} (birational morphisms are dominant).
\item[(b2)] follows from the fact that any essentially smooth scheme may be embedded in a smooth scheme of finite type by an ``essentially open immersion''.
\item[(b3)] We are given $i:X\to \bar X$ and $j:X\to Y$ where $X\in \Sm^{ess}$, $\bar X,Y\in \Sm$ and $i\in S_b$; we must factor $i$ and $j$ through $X\by{s} U$ with $U$ in $\Sm$ and $s$,  $U\to \bar X$ in $S_b$. We take for $U$ the smooth locus of the closure of the diagonal image of $X$ in $\bar X\times Y$. \end{itemize}
\end{proof}

\enlargethispage*{20pt}

To define $\Psi$, it is now sufficient to construct it as a functor $\Psi:S_r^{-1} \dv^\op \to S_b^{-1}\Sm^\ess$. We first construct $\Psi$ on $\dv^\op$ by extending the functor $\Psi_\dom$ of \eqref{eq4.3a} from $\field_s^\op$ to $\dv^\op$. For this, we repeat the construction given on \cite[p. 2041]{a-m}: if  $K\in \dv$ and $\sO$ is a good dvr with quotient field $K$ and residue field $E$, then the morphism $\Spec K \to \Spec \sO$ is an isomorphism in $S_b^{-1}\Sm^\ess$, hence the quotient map $\sO\to E$ induces a morphism $\Spec E\to \Spec K$. 

By Lemma \ref{l4.1}, any morphism in $\dv$ has a unique expression in the form \eqref{eq-dv}, which extends the definition of $\Psi$ to all morphisms. To show that $\Psi$ is a functor, it now suffices to check that it  converts any diagram \eqref{eq4.1} into a commutative diagram, which is obvious by going through its construction. Finally, $\Psi$ factors through $S_r^{-1}\dv^\op$ thanks to Theorem \ref{sb=sr}. It is now clear that the dual of $\Psi$ gives back the Asok-Morel functor \eqref{eq.am}.

As in  \S \ref{mor-rat}, we associate to a rational map $f$ between smooth varieties a morphism in $S_b^{-1}\Sm$, still denoted by $f$. We need the following analogue  of Lemma \ref{l3.3}: 

 \begin{prop}\label{l3.3a} Let $\lambda:K\tto L$ be a morphism in $\dv$. Then, for any smooth model $X$ of $K$ on which $\lambda$ is finite, we have $\Psi(\lambda) = st^{-1}f$, where $f:\Psi(L) \ttto X$ is the corresponding rational map and $s:U\inj \Psi(K)$, $t:U\inj X$ are open immersions of a common open subset $U$.
\end{prop}

\begin{proof} We proceed by induction on the length $n$ of a chain \eqref{eq-dv}: If $n=0$ the claim is trivial and if $n=1$ it is true by construction. If $n>1$, break $\lambda$ as 
\[K \overset{\lambda_1}{\tto} K_{n-1} \overset{\lambda_2}{\tto} K_n\inj  L\]
where $\lambda_1$ has rank $n-1$ and $\lambda_2$ has rank $1$. We now apply Lemma \ref{lvaquie}: since $\lambda$ is finite on $X$, so is $\lambda_1$, and if we write $Z$ for the closure of $c_X(\lambda_1)$, then $z=c_X(\lambda)=c_Z(\lambda_2)$. If $n=0$ the claim is trivial and if $n=1$ it is true by construction. If $n>1$,  Theorem \ref{t0} provides us with $\pi:X_{n-1}\to Z$, $X_{n-1}$ smooth with function field $K_{n-1}$ on which $\lambda_2$ has a centre of codimension $1$. Then we have a diagram
\[\xymatrix{
\Psi(K) & \Psi(K_{n-1}) &\Psi(K_n)\\
U\ar[u]^s\ar[d]_t & U_{n-1}\ar[u]^{s_{n-1}}\ar[d]_{t_{n-1}} & U_n\ar[u]^{s_{n}}\ar[d]_{t_{n}}\\
X & X_{n-1}\ar[d]^\pi&X_n\ar[l]_{g}\\
& Z\ar[ul]_i
}\]
where $i$ is the closed immersion $Z\inj X$, $s,s_{n-1},s_n,t,t_{n-1}$ are open immersions and $g$ is the  closed immersion of a smooth divisor obtained by applying Lemma \ref{l4.3} after possibliy shrinking $X_{n-1}$. Thus $(gt_n,s_n)$ represents the rational map given by the centre of $\lambda_2$ on $X_{n-1}$.  The rational map corresponding to $\lambda_1$ is represented by $(f_{n-1},s_{n-1})$ with 
\[f_{n-1}=i\pi t_{n-1}\]
and the one corresponding to $\lambda_2\lambda_1$ is represented by $(f_n,s_n)$ with 
\[f_n= i\pi gt_n\]
because this is compatible with $\lambda_2\lambda_1$ by Proposition \ref{p1} (also use the uniqueness in Lemma \ref{l2bis}). 

By induction and definition, we have
\[\Psi(\lambda_1)=st^{-1}f_{n-1}s_{n-1}^{-1},\quad \Psi(\lambda_2) = s_{n-1}t_{n-1}^{-1}gt_ns_n^{-1} \]
so we have to show that
\[st^{-1}f_{n-1}s_{n-1}^{-1}s_{n-1}t_{n-1}^{-1}gt_ns_n^{-1}=st^{-1}f_ns_n^{-1}\]
or
\[f_{n-1}t_{n-1}^{-1}gt_n=f_n=i\pi gt_n\] 
which is true because $f_{n-1}=i\pi t_{n-1}$. This concludes the proof.
\end{proof} 

\begin{rk} In this proof, there is no codimension condition on $c_X(\lambda)$. So Theorem \ref{t0} is used twice in a weak form: once, implicitly, to ensure the existence of $X$. Then a second time, to deal with $Z$. But here $\lambda_2$ is a discrete valuation of rank 1, so this special case can perhaps already be obtained by examining the proof of \cite[Th. 31]{zs} (which may have been a source of inspiration for \cite{knaf-kuhlmann}.)
\end{rk}

\begin{lemma}\label{l5.4} a) Let $\dv_*$ be the full subcategory of $\dv$ whose objects are in $\place_*$. Then the diagram of functors
\[\xymatrix{
S_r^{-1}\place_*^\op\ar[r]^{\Psi_*}& S_b^{-1}\Sm^\proper_*\ar[dd]\\
S_r^{-1}\dv_*^\op\ar[u]\ar[d]\\
S_r^{-1}\dv^\op\ar[r]^\Psi& S_b^{-1}\Sm
}\]
is naturally commutative.\\
 b) Let $K,L\in \dv$ and $\lambda,\mu\in \dv(K,L)$ with the same residue field $K'\subseteq L$. Suppose that $\lambda$ and $\mu$ have a common centre on some smooth model of $K$. Then $\Psi(\lambda)=\Psi(\mu)$. 
\end{lemma}

\begin{proof} a) Same argument as in \S \ref{recap}, using the natural transformation $\Spec \Rightarrow \Sigma$ of  \eqref{eqeta}. 
b) follows from Proposition \ref{l3.3a} (compare proof of Theorem \ref{cplsm} b) in \S \ref{scplsm}). 
\end{proof}

\subsection{Homotopy of places}

\begin{defn}\label{deq1} Let $K, L \in \place$. Two places
$\lambda_0,\lambda_1: K \tto L$ are {\it elementarily
homotopic} if there exists a place $\mu: K \tto L(t)$
such that $s_i\circ\mu=\lambda_i, i=0,1,$ where $s_i: L(t)
\tto L$ denotes the place corresponding to specialisation at
$i$.
\end{defn}


The property of two places being elementarily homotopic is
preserved under composition on the right. Indeed if $\lambda_0$ and
$\lambda_1$ are elementarily homotopic and if $\mu:M\tto K$ is
another place, then obviously so are $\lambda_0\circ \mu$ and
$\lambda_1\circ \mu$. If on the other hand $\tau:L\tto M$ is
another place, then $\tau\circ \lambda_0$ and
$\tau\circ\lambda_1$ are not in general elementarily homotopic (we are indebted to Gabber for
pointing this out), as one can see for example from the uniqueness of factorisation of places   \cite[p. 10]{zs}.

Consider the equivalence relation $\h$ generated by elementary
homotopy (\cf Definition \ref{d1.5}). So $\h$ is the coarsest equivalence relation on morphisms
in $\place$ which is compatible with left and right composition and such that two
elementarily homotopic places are equivalent with respect to $\h$.

\begin{defn}[\cf Def. \protect{\ref{d1.5}}]\label{deq2} We denote by $\C$
the factor
category of $\place$ by the homotopy relation $\h$.
\end{defn}

Thus the objects of $\C$ are function fields, while the set
of morphisms consists of equivalence classes of homotopic places between
the function fields. There is an obvious full surjective functor
$$\Pi:\place\to\C.$$

The following proposition provides a more elementary description
of $S_r^{-1}\place$ and of the localisation functor.

\begin{prop}\label{leq1}
There is a unique isomorphism of categories 
\[\C\to S_r^{-1}\place\]
which makes the diagram of categories and functors
\[\xymatrix{
&\place\ar[dl]_\Pi\ar[dr]^{S_r^{-1}}\\
\C\ar[rr]^{\sim}&& S_r^{-1}\place
}\]
commutative. In particular, the localisation functor $S_r^{-1}$ is
full and its fibres are the equivalence classes for $\h$. These results remain true when restricted to the subcategory  $\dv$.
\end{prop}

\begin{proof}\footnote{See also \protect\cite[Remark 1.3.4]{GP} for a closely related
statement.} We first note that any two homotopic places become equal in
$S_r^{-1}\place$. Clearly it suffices to prove this when they are
elementarily homotopic. But then $s_0$ and $s_1$ are left inverses of the
natural inclusion $i:L\to L(t)$, which becomes an isomorphism in
$S_r^{-1}\place$. Thus $s_0$ and $s_1$ become equal in
$S_r^{-1}\place$.
So the 
localisation functor
$\place\to S_r^{-1}\place$ canonically factors through $\Pi$ into a
functor $\C~~\longrightarrow~~ S_r^{-1}\place$.

On the other hand we claim that, with the above notation, $i\circ
s_0:L(t)\tto L(t)$ is homotopic to $1_{L(t)}$ in $\place$. Indeed they
are elementarily homotopic via the place $L(t)\tto L(t,s)$ (in this case
an inclusion) that is the identity on $L$ and maps $t$ to $st$. Hence the
projection functor $\Pi$ factors as $S_r^{-1}\place\to \C$,
and it is plain that this functor is inverse to the previous
one.

The claim concerning $\dv$ is clear since the above proof only used good dvr's.
\end{proof}

\subsection{Another equivalence of categories} In this subsection, we study the ``fibres'' of the functor  $\Psi$ of \eqref{eqPsi} in the light of the last condition of \cite[\S 6]{a-m}, (A3). Using Proposition \ref{leq1}, we may view $\Psi$ as a functor
\[\Psi:(\dv/\h)^\op\to S_b^{-1}\Sm.\]

Condition (A3) of \cite[\S 6]{a-m} for a functor $\sS\in \sF_F^r-\Set$ requires that for any $X\in \Sm$ with function field $K$, for any $z\in X^{(2)}$ (with separably generated residue field) and for any $y_1,y_2\in X^{(1)}$ both specialising to $z$, the compositions
\[\sS(K)\to \sS(F(y_i))\to \sS(z), \quad i=1,2\]
are equal. We can interpret this condition in the present context by introducing the equivalence relation $\h_{AM}$ in $\dv$ generated by $\h$ and the following relation $\equiv$:
\begin{quote}\it 
Given $K, L\in \dv$ and two places $\lambda_1,\lambda_2:K\tto L$ of the form
\begin{equation}\label{eq.AM}
\begin{CD}
K@>\mu_1>> K_1@>\nu_1>> L\\
K@>\mu_2>> K_2@>\nu_2>> L
\end{CD}
\end{equation}
where $\mu_1,\nu_1,\mu_2,\nu_2$ stem from good dvr's,  $\lambda_1\equiv \lambda_2$ if $\lambda_1$ and $\lambda_2$ have a common centre with residue field $L$ on some smooth model of $K$. 
\end{quote}



By Yoneda's lemma, \cite[Th. 6.1.7]{a-m} then yields an equivalence of categories
\begin{equation}\label{eq6.1}
(\dv/\h_{AM})^\op\iso S_b^{-1}\Sm.
\end{equation}

Here we implicitly used Lemma \ref{l5.4} b) and Theorem \ref{t0} to see that the functor $(\dv/\h)^\op\to S_b^{-1}\Sm$ factors through $\h_{AM}$,  as well as the following lemma:

\begin{lemma} Let $\psi:\sC\to \sD$ be a functor such that the induced functor $\psi^*:\sD^\vee\to \sC^\vee$ is an equivalence of categories. Then $\psi$ is fully faithful, hence an equivalence of categories if it is essentially surjective.
\end{lemma}

(Note that the essential surjectivity of \eqref{eq6.1} is obvious.)

\begin{proof} By \cite[I.5.3]{SGA4}, $\psi^*$ has a left adjoint $\psi_!$ which commutes naturally with $\psi$ via the Yoneda embeddings $y_\sC,y_\sD$. Since $\psi^*$ is an equivalence of categories, so is $\psi_!$; the conclusion then follows from the full faithfulness of $y_\sC$ and $y_\sD$.
\end{proof}

We now slightly refine the equivalence \eqref{eq6.1}:

\enlargethispage*{30pt}

\begin{thm}\label{t4.1} a) The functor $\Psi$ induces an equivalence of categories:
\[\overline\Psi:(\dv/\h')^\op\iso S_b^{-1}\Sm\]
where $\h'$ is the equivalence relation generated by $\h$ and the relation \eqref{eq.AM} restricted to the tuples $(\mu_1,\nu_1,\mu_2,\nu_2)$ such that  $\nu_2$ is of the form  $s_0:L(t)\tto L$ (specialisation at $0$).  In particular, $\Psi$ is full.\\ 
b) Any morphism of $\dv/\h'$ may be written in the form $\iota^{-1}f$ for $f$ a morphism of the form \eqref{eq.elem} and $\iota$ a rational extension of function fields.
\end{thm}

\begin{proof} a) Let us show that $\h'=\h_{AM}$. Starting from $K$, $\lambda_1$ and $\lambda_2$ as above, we get a smooth model  $X$ of $K$ and $z,y_1,y_2\in X$ with $z$ of codimension $2$, such that $\mu_i$ is specialisation to $y_i$ and $\nu_i$ is specialisation from $y_i$ to $z$. Shrinking, we may assume that the closures $Z,Y_1,Y_2$ of $z,y_1,y_2$ are smooth. Let $X'=\Bl_Z(X)$ be the blow-up of $X$ at $Z$ and let $Y'_1,Y'_2$ be the proper transforms of $Y_1$ and $Y_2$ in $X'$. The exceptional divisor $P$ is a projective line over $Z$ and $Z_i=P\cap Y'_i$ maps isomorphically to $Z$ for $i=1,2$. We then get new places
\begin{equation}\label{eq.AM2}
\begin{CD}
\lambda'_1:K@>\mu'>> M@>\nu'_1>> L\\
\lambda'_2:K@>\mu'>> M@>\nu'_2>> L
\end{CD}
\end{equation}
where $M=F(P)$, $L=F(Z)$ and $\lambda'_i\equiv \lambda_i$.

In $\dv/\h\simeq S_r^{-1}\dv$, the morphisms $\nu'_1$ and $\nu'_2$ are inverse to the rational extension $L\inj L(t)\simeq M$, hence are equal, which concludes the proof that $\h'=\h_{AM}$.  The fullness of $\Psi$ now follows from the obvious fullness of $\dv\to \dv/\h'$.

The argument in the proof of a) shows in particular that any composition $\nu\circ \mu$ of two good dvr's is equal in $\dv/\h'$ to such a composition in which $\nu$ is inverse to a purely transcendental extension of function fields:  b) follows from this by induction on the number of dvr's appearing in a decomposition \eqref{eq-dv}. 
\end{proof}

\begin{rks}  
%
1) Via $\overline\Psi$, Theorem \ref{t4.1}  yields a structural result for morphisms in $S_b^{-1} \Sm$,  closely related to Proposition \ref{p3.4} c) but weaker.  See however Theorem \ref{t6.1} below.

2) We don't know any example of an object in $\sF_F^r-\Set$ which verifies (A1), (A2) and (A4) but not (A3): it would be interesting to exhibit one.
\end{rks}

\subsection{$R$-equivalence} Recall the following definition of Manin:

\begin{defn}\label{dreq} a) Two
  rational points $x_0,x_1$ of a (separated) $F$-scheme $X$ of finite type are
   \emph{directly
   $R$-equivalent} if there is a rational map $f:\P^1\ttto X$ defined
   at $0$ and $1$ and such that $f(0)=x_0$, $f(1)=x_1$.\\
b) $R$-equivalence on $X(F)$ is the equivalence relation generated by
   direct $R$-equivalence.
\end{defn}

Recall that, for any $X,Y$, we have an isomorphism
\begin{equation}\label{prodr}
(X\times Y)(F)/R\iso X(F)/R\times Y(F)/R.
\end{equation}

The proof is easy.

If $X$ is proper, any rational map as in Definition \ref{dreq} a) extends to a morphism; the notion of $R$-equivalence is therefore the same as Asok-Morel's notion of $\A^1$-equivalence in \cite{a-m}. Another of their results is then, in the above language:

\begin{thm}[\protect{\cite[Th. 6.2.1]{a-m}}]\label{ta-m} Let $X$ be a proper $F$-scheme. Then the rule
\[Y\mapsto X(F(Y))/R\]
defines a presheaf of sets $\Upsilon(X)\in(S_b^{-1}\Sm)^\vee$.
\end{thm}

Note that $X\mapsto \Upsilon(X)$ is obviously functorial. 

The main point is that $R$-equivalence classes on $X$ specialise well with respect to good discrete valuations. Such a result was originally indicated by Koll\'ar \cite[p. 1]{kollarsp} for smooth proper schemes over a discrete valuation ring $R$, and proven by Madore \cite[Prop. 3.1]{madore} for projective schemes over $R$.  Asok and Morel's proof uses Lipman's resolution of $2$-dimensional schemes as well as a strong factorisation result of Lichtenbaum; as hinted by Colliot-Th\'el\`ene,  it actually suffices to use the more elementary results of \v Safarevi\v c \cite[Lect. 4, Theorem p. 33]{safa}.

Let $X$ be proper and smooth. Its generic point $\eta_X\in X(F(X))$ defines by Yoneda's lemma a morphism of presheaves
\begin{equation}\label{eq4.5}
\eta(X):y(X)\to \Upsilon(X)
\end{equation}
where $y(X)\in (S_b^{-1}\Sm)^\vee$ is the presheaf of sets represented by $X$; $\eta:X\mapsto \eta(X)$ is clearly a morphism of functors.

\begin{thm}\label{t6.1} $\eta$ is an isomorphism of functors.
Explicitly: for $Y\in \Sm$, $\eta(X)$ induces an isomorphism
\begin{equation}\label{eq4.2}
S_b^{-1}\Sm(Y,X)\iso X(F(Y))/R.
\end{equation}
\end{thm}

\begin{proof} Since $K\mapsto X(K)$ is a functor on $\dv^\op$ (compare \cite[Lemma 6.2.3]{a-m}), we have a commutative diagram for any $Y\in \Sm$: 
\begin{equation}\label{eq5.1}
\xymatrix{
\dv^\op(F(Y),F(X))\ar[r]^{\tilde\eta}\ar[d]_{\Psi}& X(F(Y))\ar[d]^\pi\ar[dl]_\epsilon\\
S_b^{-1}\Sm(Y,X)\ar[r]^\eta& X(F(Y))/R.
}\end{equation}

Here $\tilde\eta$ is obtained from $\eta_X$ by Yoneda's lemma in the same way as \eqref{eq4.5}, $\Psi$ is (obtained from) the functor of  \eqref{eqPsi}, $\pi$ is the natural projection and $\epsilon$ associates to a rational map its class in $S_b^{-1}\Sm(Y,X)$  (see comment just before Proposition \ref{l3.3a}). Here the commutativity of the top triangle follows from Proposition  \ref{l3.3a}. The surjectivity of $\pi$ shows the surjectivity of $\eta$. Note further that $\Psi$ is surjective by Theorem \ref{t4.1} a). This shows that $\epsilon$ is also surjective.

To conclude, it suffices to show that $\epsilon$ factors   through $\pi$, thus yielding an inverse to $\eta$. If $x_0,x_1\in X(F(Y))$ are directly $R$-equivalent, up to shrinking $Y$ we have a representing commutative diagram
\[\xymatrix{
Y \ar[r]^{x_0}_{x_1}\ar[d]_{s_0}^{s_1}& X\\
\P^1_Y\ar[ur]_h
}\]
with $s_0,s_1$ the inclusions of $0$ and $1$. But if we view $X(F(Y))$ and $S_b^{-1}\Sm(Y,X)$ as functors of $F(Y)\in \dv^\op$ (the second one via $\Psi$), then $\epsilon$ is checked to be a natural transformation:  indeed, this is easy in the case of an inclusion of function fields and follows from the properness of $X$ in the case of a good dvr. Hence we get $\epsilon(x_0)=\epsilon(x_1)$ since $S_b^{-1}\Sm(Y,X)\allowbreak\iso S_b^{-1}\Sm(\P^1_Y,X)$ by Theorem \ref{sb=sr}.
\end{proof}

\begin{cor}\label{c5.1} The functor $\theta:S_b^{-1}\Sm_*^\proper\to S_b^{-1} \Sm$ of \eqref{eq.natural} is fully faithful.
\end{cor}

\begin{proof}
For $X,Y\in \Sm^\proper_*$, we have a commutative diagram similar to \eqref{eq5.1} replacing $\dv$ by $\place_*$ and $\Sm$ by $\Sm^\proper_*$.  The map $\eta_*$ corresponding to $\eta$ is obtained from \eqref{eq5.1} by composition, while the map $\tilde \eta_*$ corresponding to $\tilde \eta$ exists because $K\mapsto X(K)$ is a functor on $\place^\op$ by the valuative criterion of properness. Further,  the  map corresponding to $\epsilon$ is well-defined thanks to Proposition \ref{p4a} b) and the top triangle commutes thanks to Lemma \ref{l3.3}.  The natural map from this diagram to \eqref{eq5.1} yields a commutative diagram thanks to Lemma \ref{l5.4}. Moreover, the map corresponding to $\Psi$ is surjective thanks to Theorem \ref{cplsm} a). The same reasoning  as above then shows that $\eta_*$ is bijective: we just have to replace ``up to shrinking $Y$'' by ``up to replacing $Y$ by a birationally equivalent smooth projective variety'', using the graph trick and the definition of $\Sm_*^\proper$.  The graph trick can also be used to reduce the verification that $\epsilon$ is natural to the case where the rational maps involved are in fact morphisms. Hence the conclusion.
\end{proof}

\begin{rk} One could replace $\Sm_*^\proper$ by  $\Sm^\proj_*$ in Corollary \ref{c5.1}
, thus getting a full embedding $S_b^{-1}\Sm^\proj_* \inj S_b^{-1}\Sm^\proper_*$.
\end{rk}

The following corollary generalises \cite[Prop. 10]{ct-s} to any characteristic:

 \begin{cor}\label{c6.1} Let $s:Y\ttto X$ be a rational map, with $X,Y\in \Sm^\proper$, Then $s$ induces an map $s_*:Y(K)/R\to X(K)/R$ for any  $K\in \dv$. Moreover, $s_*$ is a bijection for any $K\in\dv$ if and only if
 the morphism $\tilde s$ associated to $s$ in $S_b^{-1} \Sm$ (see comment just before Proposition \ref{l3.3a})  is an isomorphism.\\
In particular, $s_*:Y(K)/R\iso X(K)/R$ for any  $K\in \dv$ when $s$ is dominant and the field extension $F(Y)/F(X)$ is rational.
\end{cor}

\begin{proof} The morphism $\tilde s$ induces a morphism $S_b^{-1}(U,Y)\to S_b^{-1}(U,X)$ for any $U\in \Sm$, hence the first claim follows from Theorem \ref{t6.1}. ``If'' is obvious, and ``only if'' follows from Yoneda's lemma. Finally,  Theorem \ref{sb=sr}  implies that $\tilde s$ is an isomorphism under the last hypothesis on $s$, hence the conclusion.
\end{proof}

See Theorem \ref{c3.5} for a further generalisation.

\enlargethispage*{30pt}

\subsection{Coronidis loco} Let us go back to the diagram in Lemma \ref{l5.4} a).  Let $\h'_*$ be the equivalence relation on $\dv_*$ defined exactly as $\h'$ on $\dv$ (using objects of $\dv_*$ instead of objects of $\dv$).
On the other hand, let $\h''$ be the equivalence relation on $\place_*$ generated by $\h$ and

\begin{quote}\it For $\lambda,\mu:K\tto L$, 
$\lambda\sim \mu$ if $\lambda$ and $\mu$ have a common centre on some model $X\in \Sm_*^\proper$ of $K$.
\end{quote}

Clearly, the restriction of $\h''$ to $\dv_*$ is coarser than $\h'$; 
hence, using  Theorem \ref{cplsm} b) and Proposition \ref{leq1}, we get an induced naturally commutative diagram:
\[\xymatrix{
(\place_*/\h'')^\op\ar[r]^{\overline\Psi_*}& S_b^{-1}\Sm^\proper_*\ar[dd]^\theta\\
(\dv_*/\h'_*)^\op\ar[u]^a\ar[d]_b\\
(\dv/\h')^\op\ar[r]^{\overline\Psi}& S_b^{-1}\Sm.
}\]

In this diagram, $\overline\Psi_*$  is full and essentially surjective by Theorem 
\ref{cplsm} a), $\overline \Psi$ is an equivalence of categories by Theorem \ref{t4.1} a) and $\theta$ is fully faithful by Corollary \ref{c5.1}. Moreover,  $a$ is full by Lemma \ref{l2bis} and the proof of Lemma \ref{l6}, and essentially surjective by defnition. All this implies:

\begin{thmr}\label{tcor-loc}  If $\car k=0$, all functors in the above diagram are equivalences of categories.
\end{thmr}

\begin{proof} If $\car k=0$, $\dv_*=\dv$ hence $b$ is the identity functor.  In view of the above remarks, the diagram then shows that $a$ is faithful, hence an equivalence of categories. It follows that  $\overline\Psi_*$ is also an equivalence of categories. Finally $\theta$ is essentially surjective, which completes the proof.
\end{proof}

As an application, we get a generalisation of the specialisation theorem to arbitrary places (already obtained in \cite[Cor. 7.1.2]{birat}):

\begin{corr} Suppose $\car F=0$. Let $X\in \Var^\proper$, $K,L\in \place_*$ and $\lambda:K\tto L$ be a place. Then $\lambda$ induces a map
\[\lambda_*:X(K)/R\to X(L)/R.\]
If $\mu:L\tto M$ is another place, with $M\in \place_*$, then $(\mu\lambda)_* = \mu_*\lambda_*$.
\end{corr}

\begin{proof} By Theorem \ref{ta-m}, $K\mapsto X(K)/R$ defines a presheaf on $(\dv/h')^\op$, which extends to a presheaf on $(\place_*/\h'')^\op$ by Theorem \ref{tcor-loc}.
\end{proof}

\section{Linear connectedness of exceptional loci}\label{sbr}


\subsection{Linear connectedness} We have the
  following definition of Chow \cite{chow}:

\begin{defn}\label{dreq-Chow} A (separated) $F$-scheme $X$ of finite type is \emph{linearly connected} if any two points of $X$ (over a
universal domain) may be joined by a chain of rational curves.
\end{defn}

Linear connectedness is closely related to the notion of rational chain-connectedness of Koll\'ar et al., for which we refer to \cite[p. 99, Def. 4.21]{debarre}. In fact:

\begin{prop}\label{r5.1}  The following conditions are equivalent:
\begin{thlist}
\item $X$ is linearly connected.
\item For any algebraically
closed extension $K/F$, $X(K)/R$ is reduced to a point.\\
If  $X$ is a proper $F$-variety, these conditions are equivalent to:
\item $X$ is rationally chain-connected. 
\end{thlist}
\end{prop}

\begin{proof} (ii) $\Rightarrow$ (i) is obvious by definition (take for $K$ a universal domain). For the converse, let $x_0,x_1\in X(K)$. Then $x_0$ and $x_1$ are defined over some finitely
generated subextension $E/F$. By assumption, there exists a universal domain
$\Omega\supset E$ such that $x_0$ and $x_1$ are $R$-equivalent in $X(\Omega)$. Then the algebraic closure $\bar E$ of $E$ embeds into $\Omega$ and $K$. If $x_0$ and $x_1$ are $R$-equivalent in $X(\bar E)$, so are they in $X(K)$; this reduces us to the case where $K\subseteq \Omega$.

Let $\gamma_1,\dots,\gamma_n:\P^1_\Omega\ttto X_\Omega$ be a chain of rational curves linking
$x_0$ and $x_1$ over $\Omega$. Pick a finitely generated extension $L$ of $K$ over which all the
$\gamma_i$ are defined. 

We may write $L=K(U)$ for some $K$-variety $U$. Then the $\gamma_i$ define rational maps
$\tilde \gamma_i:U\times \P^1\ttto X$. Since each $\gamma_i$ is defined at $0$ and $1$ with $\gamma_i(1)=\gamma_{i+1}(0)$, we may if needed shrink $U$ so that the domains of definition of all the $\tilde \gamma_i$ contain $U\times\{0\}$ and $U\times\{1\}$. Moreover, these restrictions coincide in the same style as above, since they do at the generic point of $U$.  Pick a rational point $u\in U(K)$: then the fibres of the $\tilde\gamma_i$ at $u$ are rational curves defined over $K$ that link $x_0$ to $x_1$.

A rationally chain connected $F$-scheme is a proper variety by definition; then (i) $\iff$ (iii) if $F$ is uncountable by \cite[p. 100, Remark
4.22 (2)]{debarre}. On the other hand, the property of linear connectedness is clearly invariant
under algebraically closed extension, and the same holds for  rational chain-connectedness by  \cite[p. 100, Remark 4.22 (3)]{debarre}. Thus (i) $\iff$ (iii) holds in general.
\end{proof}

We shall discuss the well-known relationship with
rationally connected varieties in \S \ref{slc}.

Proposition \ref{r5.1} suggests the following definition:

\begin{defn}\label{dreq-Chow-s} An $F$-scheme $X$ of finite type is \emph{strongly linearly connected} if  $X(K)/R=*$ for any separable extension $K/F$.
\end{defn}

\subsection{Theorems of Murre, Chow, van der Waerden and Gabber}\label{5.3}
We start with the following not so well-known but nevertheless basic theorem
of Murre \cite{murre}, which was later rediscovered by Chow and van der Waerden
\cite{chow,vdw}.  

\begin{thm}[Murre, Chow, van der Waerden]\label{murre} Let $f:X\to Y$ be
  a projective birational morphism of $F$-varieties and $y\in Y$ be a
  smooth rational
  point. Then the fibre $f^{-1}(y)$ is linearly connected. In particular, by Proposition 
\ref{r5.1}, $f^{-1}(y)(K)/R$ is reduced to a point for any algebraically closed extension
$K/F$.
\end{thm}

For the sake of completeness, we give the general statement of Chow, which does not require a
base field:

\begin{thm}[Chow]\label{chow} Let $A$ be a regular local ring and $f:X\to \Spec A$ be a
projective birational morphism. Let $s$ be the closed point of $\Spec A$ and $F$ its residue
field. Then the special fibre $f^{-1}(s)$ is linearly connected (over $F$).
\end{thm}

Gabber has recently refined these theorems:

\begin{thm}[Gabber]\label{chowstrong} Let $A,X,f,s,F$ be as in Theorem \ref{chow}, but assume
only that $f$ is proper. Let $X_\reg$ be the regular locus of $X$ and $f^{-1}(s)^\reg=
f^{-1}(s) \cap X_\reg$, which is known to be open in $f^{-1}(s)$. Then, for any extension
$K/F$, any two points of
$f^{-1}(s)^\reg(K)$ become $R$-equivalent in $f^{-1}(s)(K)$.\\
In particular, if $X$ is regular, then $f^{-1}(s)$ is strongly linearly connected.
\end{thm}

See Appendix \ref{gabber} for a proof of Theorem \ref{chowstrong}.

\begin{thm}[Gabber \protect{\cite{gabber}}]\label{gabber2} If $F$ is a field, $X$ is a regular irreducible $F$-scheme
of finite type and
$K/F$ a field extension, then the map
\[ \lim X'(K)/R\to X(K)/R\] 
has a section, which is contravariant in $X$ and covariant in $K$. The limit is over the proper
birational $X'\to X$.
\end{thm}

\subsection{Applications} The following theorem extends part of Corollary \ref{c6.1} to a relative setting: 

\begin{thm} \label{c3.5} 
a) Let $s:Y\to X$ be in $S_b^p$, with $X,Y$ regular. Then the induced map $Y(K)/R\to X(K)/R$ is
bijective for any field extension $K/F$. If $K$ is algebraically closed, the hypothesis ``$Y$ regular'' is not necessary.\\
b) Let $f: Y\ttto Z$ be a rational map with $Y$
regular and $Z$ proper. Then there is an induced map $f_*: Y(K)/R \to Z(K)/R$, which depends
functorially on $K/F$.
\end{thm}

\begin{proof} a) As in the proof of Proposition \ref{p3.4} a),
it suffices to deal with $K=F$. By this proposition, we have to show injectivity.

We assume that $s\in S_b^p$. Let $y_0,y_1\in Y(F)$. Suppose that $s(y_0)$ and $s(y_1)$ are
$R$-equivalent. We want to show that $y_0$ and $y_1$ are then $R$-equivalent. By definition, $s(y_0)$ and $s(y_1)$ are connected by a chain of direct
$R$-equivalences. Applying Proposition \ref{p3.4} a), the intermediate rational points lift
to $Y(F)$. This reduces us to the case where $s(y_0)$ and $s(y_1)$ are
directly $R$-equivalent.

Let $\gamma:\P^1\ttto X$ be a rational map defined at $0$ and $1$ such
that $\gamma(i)=s(y_i)$. Applying Proposition \ref{p3.4} a) with $K=F(t)$, we get that
$\gamma$ lifts to a rational map $\tilde \gamma:\P^1\ttto Y$. Since
$s$ is proper, $\tilde\gamma$ is still defined at $0$ and $1$. Let
$y'_i=\tilde\gamma(i)\in Y(F)$: then $y_i,y'_i\in s^{-1}(s(y_i))$. If $F$ is algebraically
closed, they are $R$-equivalent by Theorem \ref{murre}, thus $y_0$ and $y_1$ are
$R$-equivalent. If $F$ is arbitrary but $Y$ is regular, then we appeal to Theorem
\ref{chowstrong}.

b) By the usual graph trick, as $Z$ is proper, we can
resolve $f$ to get a morphism 
\[\xymatrix{
&\tilde{Y}\ar[dl]_p\ar[dr]^{\tilde{f}}\\
Y && Z
}\]
such that $p$ is a proper birational morphism. By Theorem \ref{gabber2}, the map
$p_*: \tilde{Y}(K)/R \to Y(K)/R$ has a section which is ``natural" in $p$ (\ie when we take a
finer $p$, the two sections are compatible). The statement follows.
\end{proof}

\begin{rk}
Concerning Theorem \ref{chowstrong}, Fakhruddin pointed out that $f^{-1}(s)$ is in
general not strongly linearly connected, while Gabber pointed out that $f^{-1}(s)^\reg(F)$ may
be empty even if $X$ is normal, when $F$ is not algebraically closed. Here is Gabber's example:
in dimension $2$, blow-up the maximal ideal of $A$ and then a non-rational point of the special
fiber, then contract the proper transform of the special fiber. Gabber also gave examples 
covering Fakhruddin's remark: suppose $\dim A=2$ and start from $X_0=$ the blow-up of $\Spec A$
at $s$. Using \cite{ferrand}, one can ``pinch" $X_0$ so as to convert a non-rational closed
point of the special fibre into a rational point. The special fibre of the resulting $X\to
\Spec A$ is then a singular quotient of $\P^1_F$, with two $R$-equivalence classes. He
also gave a normal example \cite{gabber}.
\end{rk}

\section{Examples, applications and open questions}\label{s7}

In this section, we put together some concrete applications of the above results
and list some open questions.

\subsection{Composition of $R$-equivalence classes} As a by-product of Theorem \ref{t6.1}, one gets for three smooth proper varieties $X,Y,Z$ over a field of characteristic $0$ a composition law
\begin{equation}\label{eq9.1}
Y(F(X))/R\times Z(F(Y))/R\to Z(F(X))/R
\end{equation}
which is by no means obvious.  As a corollary, we have:

\begin{cor}\label{c5.4.10}
Let $X$ be a smooth proper variety with function field $K$. Then
$X(K)/R$ has a structure of a monoid with $\eta_X$ as the identity
element.\qed
\end{cor}

\subsection{$R$-equivalence and birational functors} Here is a more concrete reformulation of part of Theorem
\ref{t6.1}:

\begin{cor} Let
\[P:\Sm^\proper(F)\to \sA\] 
be a functor to some category
$\sA$. Suppose that $P$ is a birational functor. Then
$R$-equi\-va\-lence
classes act on $P$: if $X,Y$ are two smooth projective varieties, any class
$x\in X(F(Y))/R$ induces a morphism $x_*:P(Y)\to P(X)$. This assignment is
compatible with the composition of $R$-equivalence classes from
\eqref{eq9.1}.\\
In particular, for two morphisms $f,g:X\to Y$, $P(f)=P(g)$ as soon as
$f(\eta_X)$ and
$g(\eta_X)$ are $R$-equivalent.
\end{cor}

Theorem \ref{t6.1} further says that $R$-equivalence is
``universal" among birational functors.






\subsection{Algebraic groups and $R$-equivalence} 
As a  special case of Corollary \ref{c5.4.10}, we consider a connected algebraic group
$G$ defined over $F$. Recall that for any extension $K/F$, the set $G(K)/R$
is in fact a group. Let $\bar G$ denote a smooth compactification
of $G$ over $F$ (we assume that there is one). It is 
known (P. Gille, \cite{gi}) that the natural map $G(F)/R \to \bar
G(F)/R$ is an isomorphism  if $F$ has characteristic zero and $G$ is
reductive.

Let $K$ denote the function field $F(G)$. By the above corollary, there is
a composition law $\circ$ on $\bar G(K)/R$. On the other hand, the
multiplication
morphism 
$$
m : G \times G \to G
$$
considered as a rational map on $\bar G \times \bar G$ induces a product
map (Theorem \ref{c3.5})
$$
\bar G(K)/R \times \bar  G(K)/R \to \bar G(K)/R
$$
which we denote by $(g,h)\mapsto g\cdot h$; this is clearly compatible
with
the corresponding product map on $G(K)/R$ obtained using the
multiplication homomorphism on $G$. Thus we have two composition laws on
$\bar G(K)/R$.

The following lemma is a formal consequence of Yoneda's lemma:

\begin{lemma}\label{l5.4.11}
Let $g_1, g_2, h \in \bar G(K)/R$. Then we have $(g_1\cdot g_2)\circ
h=(g_1\circ
h)\cdot (g_2\circ h)$.\qed
\end{lemma}

In particular, let us take $G=SL_{1,A}$, where $A$ is a central simple
algebra over $F$. It is then known that $G(K)/R\simeq SK_1(A_K)$ for any
function field $K$. If $\car F=0$, we may use Gille's theorem and find
that, for $K=F(G)$, $SK_1(A_K)$ admits a second composition law with unit
element the generic element, which is distributive on the right with
respect to the multiplication law. However, it is not distributive on the left in
general: 

Note that the natural map $\Hom(\Spec F,\bar G)=\bar G(F)/R\to \bar G(K)/R=\Hom(\bar G,\bar G)$
is split injective, a retraction being induced by the unit section $\Spec F\to G\to \bar G$.
Now let
$g\in G(F)$;  for any $\phi\in G(K)= \Rat(G,G)$, we clearly have $[g]\circ
[\phi]= [g]$. In particular, $[g]\circ ([\phi]\cdot [\phi'])\ne ([g]\circ [\phi])\cdot ([g]\circ
[\phi'])$ unless $[g]=1$. (This argument works for any group object in a category with a final
object.)

\subsection{Kan extensions and $\Pi_1$}\label{kan} Let  $\Sm_{**}$ denote the full subcategory of $\Sm$ given by those smooth varieties which admit a smooth proper compactification: then the functor $\theta$ of Corollary \ref{c5.1} induces an equivalence of categories $S_b^{-1}\Sm^\proper\iso S_b^{-1}\Sm_{**}$. Suppose we are given a functor $F:\Sm\to \sC$ whose restriction to $\Sm^\proper$ is birational. We then get an induced functor $\bar F:S_b^{-1} \Sm_{**}\to \sC$ plus a natural transformation
\[\rho_X:F(X)\to \bar F(X)\]
for any $X\in \Sm_{**}$. 

To construct $\bar F$, we set
\[\bar F(X) = \lim_{\bar X} F(\bar X)\]
where the limit is on the category of open immersions $j: X\inj \bar X$ with $\bar X\in \Sm^\proper$: this is an inverse limit of isomorphisms, hence makes sense without any hypothesis on $\sC$ and may be computed by taking any representative $\bar X$. To construct $\rho_X$, an open immersion $j: X\inj \bar X$  as above yields  a map $F(X)\by{F(j)} F(\bar X)\simeq \bar F(X)$, and one checks that this does not depend on the choice of $j$. This is an instance of a \emph{right Kan extension} \cite[Ch. X, \S 3, Th. 1]{maclane}.

We may apply this to $F=\Pi_1$, the fundamental groupoid\footnote{Rather than fundamental group, to avoid the choices of base points.} (here $\sC$ is the category of groupoids): the required property is \cite[Exp. X, Cor. 3.4]{sga1}. As an extra feature, we get that the universal transformation $\rho$ is an epimorphism, because $\Pi_1(U)\surj \Pi_1(X)$  if $U\inj X$ is an open immersion of smooth schemes. Thus, $\Pi_1(X)$ has a ``universal birational quotient'' which is natural in $X$.

As another application, we get that for $X$ smooth and proper, the ``section map'' (subject to a famous conjecture of Grothendieck when $X$ is a curve)
\begin{equation}\label{eq.sect}
X(F) \to \Hom_{\Pi_1(\Spec F)}(\Pi_1(\Spec F),\Pi_1(X))
\end{equation}
factors through $R$-equivalence. On the other hand, if $X$ is projective and $Y$ is a smooth hyperplane section, then $\Pi_1(Y)\iso \Pi_1(X)$ as long as $\dim X>2$ by \cite[Exp. XII, Cor. 3.5]{sga2}; so there are more morphisms to invert if one wishes to study \eqref{eq.sect} for $\dim X>1$ by the present methods.

\subsection{Strongly linearly connected smooth proper varieties}\label{slc}

One natural question that arises is the following: characterise morphisms
$f:X\to Y$ between smooth proper varieties which become invertible in
the category $S_b^{-1}\Sm^\proper$, or equivalently in $S_b^{-1}\Sm$ by Corollary \ref{c5.1}. Here we shall study this question only
in the
simplest case, where $Y=\Spec F$. 

\begin{thm}\label{tlin_con} a) Let $X$ be a smooth proper variety over $F$. Consider the
  following conditions:
\begin{enumerate}
\item $p:X\to \Spec F$ is an isomorphism in $S_b^{-1}\Sm$.
\item $p$ is an isomorphism in $S_r^{-1}\Sm$.
\item For any separable extension $E/F$, $X(E)/R$ has one element (\ie $X$ is strongly linearly connected
according to Definition \ref{dreq-Chow-s}).
\item Same, for $E/F$ of finite type.
\item $X(F)\neq \emptyset$ and $X(K)/R$ has one element for $K=F(X)$.
\item $X(F)\neq \emptyset$ and, given $x_0\in X(F)$, there exists a
  chain of rational curves $(f_i:\P^1_K\to X_K)_{i=1}^n$ such that
  $f_1(0)=\eta_X$, $f_{i+1}(0)=f_i(1)$ and $f_n(1)=x_0$. Here
  $K=F(X)$ and $\eta_X$ is the generic point of $X$.
\item Same as {\rm (6)}, but with $n=1$.
\end{enumerate}
Then $(1)\iff (2)\iff (3)\iff (4)\iff (5)
\iff (6)\Leftarrow (7)$.\\
b) If $\car F=0$, $X$ satisfies Conditions $(1)-(6)$ and is projective, it is rationally
connected.
\end{thm}

\begin{proof} a) (1) $\Rightarrow$ (2) is trivial and the converse
  follows from Theorem \ref{sb=sr}. Thanks to Theorem \ref{t6.1}, (2)
  $\iff$ (4) is an easy consequence of the Yoneda lemma.
The implications $(3)\Rightarrow 
  (4)\Rightarrow (5) \Rightarrow (6)\Leftarrow (7)$ are trivial and (4)
$\Rightarrow$
  (3) is easy by a direct limit argument. To see $(6) \Rightarrow
(1)$, note that by Theorem \ref{t6.1} (6) implies that $1_X=x_0\circ p$ in
$S_b^{-1}\Sm(X,X)$, hence $p$ is an
isomorphism. 

b) This follows from Proposition \ref{r5.1} plus the famous theorem of Koll\'ar-Miyaoka-Mori  
\cite[Th. 3.10]{kollar}, \cite[p. 107, Cor. 4.28]{debarre}.
\end{proof}

\begin{rk}\label{r7.3}  The example of an anisotropic conic shows that, in (5), the
  assumption $X(F)\neq\emptyset$ does not follow from the next one.
\end{rk}

\begin{qn} In the situation of Theorem \ref{tlin_con} b), does $X$ verify condition (7)? We give a partial result in this direction in Proposition \ref{p7.4} below. (The reader may consult the first version of this paper for a non-conclusive attempt to answer this question in general.)
\end{qn}

\subsection{Retract-rational varieties} Recall that, following Saltman, $X$ (smooth but not
necessarily proper) is
\emph{retract-rational} if it contains an open subset $U$ such that $U$ is
a retract of an open subset of $\A^n$. When $F$ is infinite, this includes
the case where there exists $Y$ such that $X\times Y$ is rational, as in
\cite[Ex. A. pp.  222/223]{ct-s}.

We have a similar notion for function fields:

\begin{defn} A function field $K/F$ is \emph{retract-rational} if there
exists an integer $n\ge 0$ and two places $\lambda:K\tto
F(t_1,\dots,t_n)$, $\mu:F(t_1,\dots,t_n)\tto K$ such that
$\mu\lambda=1_K$.
\end{defn}

Note that this forces $\lambda$ to be a trivial place (\ie an inclusion of
fields). Using Lemma \ref{l6}, we easily see that $X$ is retract-rational
if and only if $F(X)$ is retract-rational.

\begin{prop}\label{p7.4} If $X$ is a retract-rational smooth variety, then
$X\iso \Spec F$ in $S_b^{-1}\Sm$. If moreover $X$ is proper and
$F$ is infinite, then $X$ verifies Condition {\rm (7)} of Theorem
\ref{tlin_con} for a Zariski dense set of points $x_0$.
\end{prop}

\begin{proof} The first statement is obvious by Yoneda's lemma. Let us prove the second: 
by hypothesis, there exist open subsets $U\subseteq
X$ and $V\subseteq \A^n$ and morphisms $f:U\to V$ and $g:V\to U$ such that
$gf=1_U$. This already shows that $U(F)$ is Zariski-dense in $X$. Let now $x_0\in U(F)$, and
let $K=F(X)$. Consider the straight line $\gamma:\A^1_K\to \A^n_K$ such that $\gamma(0)=f(x_0)$
and $\gamma(1)=f(\eta_X)$. Then $g\circ\gamma$ links $x_0$ to $\eta_X$, as desired.
\end{proof}

\begin{cor}\label{c7.4} We have the following
implications for a smooth prop\-er variety $X$ over a field $F$ of 
characteristic $0$: retract-rational
$\Rightarrow$ strongly linearly connected $\Rightarrow$ rationally
connected.
\end{cor}

\begin{proof} The first implication follows from  Theorem \ref{tlin_con} and Proposition
\ref{p7.4}; the second implication follows from the theorem of Koll\'ar-Miyaoka-Mori already quoted.
\end{proof}

\begin{rk} In characteristic $0$, if $X$ is a smooth compactification of a torus, then it
verifies Conditions {\rm (1) -- (6)} of Theorem \ref{tlin_con} if and only if it is
retract-rational, by \cite[Prop. 7.4]{ct-s2} (\ie the   first implication in the previous
corollary is an equivalence for such $X$). This may also be true by
replacing ``torus" by ``connected reductive group": at least it is so in many special cases,
see \cite[Th. 7.2 and Cor. 5.10]{gille-kneser}.

\end{rk}





\subsection{$S_r$-local objects} Recall:

\begin{defn}\label{d1.3} Let $\sC$ be a category and $S$ a family of
  morphisms of $\sC$. An object $X\in\sC$ is 
\emph{local} relatively to $S$ or \emph{$S$-local} (\emph{left closed}
in the terminology of 
\cite[Ch. 1, Def. 4.1 p. 19]{gz}) if, for any
$s:Y\to Z$ in
$S$, the map
\[\sC(Z,X)\overset{s^*}{\to}\sC(Y,X)\]
is bijective.
\end{defn}


 In this rather disappointing subsection,
we show that there are not enough of these objects. They are the exact
opposite of rationally connected varieties.

\begin{defn} A proper $F$-variety $X$ is \emph{nonrational} if
it does not carry any nonconstant rational curve (over the algebraic
closure of $F$), or equivalently if the map
\[X(\bar F)\to X(\bar F(t))\]
is bijective.
\end{defn}

\begin{lemma} a) Nonrationality is stable by product and by passing to
closed subvarieties.\\
b) Curves of genus $>0$ and torsors under abelian varieties are
nonrational.\\
c) Nonrational smooth projective varieties are minimal in the sense that
their canonical bundle is nef.
\end{lemma}

\prf a) and b) are obvious; c) follows from the Miyaoka-Mori
theorem (\cite{mori}, see also \cite[Th. 1.13]{km} or \cite[Th.
3.6]{debarre}).\qed

On the other hand, an anisotropic conic is not a nonrational variety. This is also true for some minimal models in dimension 2, even when $F$ is algebraically closed.

Smooth nonrational varieties are the local objects of
$\Sm$ with respect to $S_r$ in the sense of Definition
\ref{d1.3}: 

\begin{lemma} a) A proper variety $X$ is nonrational if and only if, for
any morphism $f:Y\to Z$ between smooth varieties such that $f\in S_r$, the
map
\[f^*:Map(Z,X)\to Map(Y,X)\]
is bijective.\\
b) A smooth proper nonrational variety $X$ is stably minimal in the
following sense: any morphism in $S_r$ with source $X$ is an isomorphism.
\end{lemma}

\prf a) Necessity is clear (take $f:\P^1\to \Spec F$). For sufficiency,
$f^*$ is clearly injective since $f$ is dominant, and we have to show
surjectivity. We may assume $F$ algebraically closed. Let
$U$ be a common open subset to $Y$ and $Z\times\P^n$ for suitable $n$. Let $\psi:Y\to X$. By
\cite[Cor. 1.5]{km} or \cite[Cor. 1.44]{debarre}, $\psi_{|U}$ extends to a
morphism
$\phi$ on $Z\times\P^n$. But for any closed point $z\in Z$,
$\phi(\{z\}\times\P^1)$ is a point, where
$\P^1$ is any line of $\P^n$. Therefore $\phi(\{z\}\times\P^n)$ is a point,
which implies that $\phi$ factors through the first projection.

b) immediately follows from a).\qed

\begin{lemma} If $X$ is nonrational, it remains nonrational over any
extension $K/F$.
\end{lemma}

\prf It is a variant of the previous one: we may assume that $F$ is
algebraically closed and that
$K/F$ is finitely generated. Let $f:\P^1_K\to X_K$. Spread $f$ to a
$U$-morphism 
$\tilde f:U\times \P^1\to U\times X$ and compose with the second
projection. Any closed point $u\in U$ defines a map $f_u:\P^1\to X$,
which is constant, hence $p_2\circ\tilde f$ factors through the first
projection, which implies that $f$ is constant.\qed

\subsection{Open questions}\label{s7.7} We finish by listing a few problems that are not
answered in this paper.

\begin{enumerate}
\item Compute Hom sets in $S_b^{-1}\Var$. In \cite[Rk. 8.11]{loc}, it is shown that the functor $S_b^{-1}\Sm\to S_b^{-1}\Var$ is neither full nor faithful and that the Hom sets are in fact completely different.
\item Compute Hom sets in $(S_b^p)^{-1}\Sm$.
\item Let $d_{\le n}\Sm$ be the full subcategory of $\Sm$ consisting of smooth varieties of dimension $\le n$. Is the induced functor $S_b^{-1}d_{\le n} \Sm\to S_b^{-1} \Sm$ fully faithful?
\item Give a categorical interpretation of rationally connected varieties.
\item Finally one should develop additional functoriality: products and internal Homs, change of base field.
\end{enumerate}

\appendix

\section{Invariance birationnelle et invariance ho\-mo\-to\-pi\-que} \label{colliot}

\hfill par Jean-Louis Colliot-Th\'el\`ene

\hfill 14 septembre 2006.

\bigskip

\noindent{\it Soit $k$ un corps. Soit $F$ un foncteur contravariant de la cat\'egorie des
$k$-sch\'emas vers la cat\'egorie des ensembles.  Si sur les morphismes $k$-birationnels de
surfaces projectives, lisses   et g\'eom\'etri\-quement connexes
 ce foncteur induit   des bijections, alors l'application $F(k) \to F(\P^1_{k})$ est une
bijection.}
 
\bigskip
 
 \noindent{\it D\'emonstration}.   Toutes les vari\'et\'es consid\'er\'ees sont des $k$-vari\'et\'es. On \'ecrit
 $F(k)$ pour $F({\rm Spec}(k))$.
 Soit $W$ l'\'eclat\'e de $\P^1\times \P^1$ en un $k$-point $M$. Les transform\'es propres des deux g\'en\'eratrices $L_{1}$ et $L_{2}$ passant par $M$ sont deux courbes  exceptionnelles  de premi\`ere esp\`ece  $E_{1}\simeq \P^1$ et $E_{2}\simeq \P^1$ qui ne se rencontrent pas. On peut donc les contracter simultan\'ement, la surface que l'on obtient est le plan projectif $\P^2$. Notons $M_{1}$ et $M_{2}$ les $k$-points de  $\P^2$ sur lesquels les courbes $E_{1}$ et $E_{2}$ se contractent.
  
  On r\'ealise facilement cette
  construction de mani\`ere concr\`ete.  Dans  $ \P^1 \times \P^1 \times \P^2$ avec coordonn\'ees multihomog\`enes
 $(u,v; w,z;X,Y,T)$  on prend pour $W$ la surface d\'efinie par l'id\'eal  $(uT-vX, wT-zY)$, 
 et on consid\`ere les deux projections $W \to \P^1\times \P^1$ et $W \to \P^2$.

  On a  un diagramme commutatif de morphismes 
 \[\begin{CD} E_{1} @>>> W     \\
 @V{\wr}VV @V{}VV \\
L_1 @>>> \P^1 \times \P^1.
\end{CD}\]

Le compos\'e de l'inclusion $L_{1} \inj \P^1 \times \P^1$ et d'une des deux projections
$\P^1 \times \P^1 \to \P^1$ est un isomorphisme. Par fonctorialit\'e, la restriction $F(\P^1\times \P^1) \to F(L_{1})$
est donc surjective.  Par fonctorialit\'e, le diagramme ci-dessus implique alors que la restriction $F(W) \to F(E_{1})$ est surjective.
 
 Consid\'erons maintenant la projection $W \to \P^2$.
 On a ici le  diagramme commutatif de morphismes
\[\begin{CD}
E_{1} @>>> W     \\
@V{}VV @V{}VV \\
M_{1} @>>> \P^2. 
\end{CD}\]

Par l'hypoth\`ese d'invariance birationnelle, on a la bijection  $F(\P^2) \allowbreak\iso  F(W)$. Donc  la fl\`eche compos\'ee $F(\P^2) \to F(W) \to F(E_{1})$
 est surjective. Mais par le diagramme commutatif ci-dessus la fl\`eche compos\'ee se factorise aussi comme
  $F(\P^2) \to  F(M_{1}) \to F(E_{1})$. Ainsi $F(M_{1}) \to F(E_{1})$, c'est-\`a-dire
 $F(k) \to  F(\P^1)$, est surjectif.
 L'injectivit\'e de $F(k) \to F(\P^1)$ r\'esulte de la fonctorialit\'e et de la consid\'eration d'un
 $k$-point sur $\P^1$.


\section{A letter from O. Gabber}\label{gabber}

\hfill June 12, 2007
\medskip

Dear Kahn,
\medskip

I discuss a proof of

\begin{thm} Let $A$ be a regular local ring with residue field
$k$, $X'\to X=Spec(A)$ a proper birational morphism, $X'_{reg}$ the regular
locus
of $X'$, $X'_s$ the special fiber of $X'$, $X'_{reg,s}= X'_s \cap X'_{reg}$,
which
is known to be open in $X'_s$, $F$ a field extension of $k$, then any two
points of
$X'_{reg,s}(F)$ are $R$-equivalent in $X'_s(F)$.
\end{thm}

The proof  I tried to sketch by joining centers of divisorial valuations has
a gap in the imperfect residue field case. It is easier to adapt the proof
by deformation of local arcs.

(1) If $Y'\to Y$ is proper surjective map between separated $k$-schemes
of finite type  whose fibers are
projective spaces then for every $F/k$, $Y'(F)/R\to Y(F)/R$ is bijective. In
particular the theorem holds if $X'$ is obtained from $X$ by a sequence of
blow-ups with regular centers.

(2) If $A$ is a regular local ring of dimension $>1$ with maximal ideal $\fm$,
$U$ an
open non empty in $\Spec(A)$, then there is $f \in \fm-\fm^2$ s.t. the generic
point
of $V(f)$ is in $U$.

This is because $U$ omits only a finite number of height $1$ primes and
there
are infinitely many possibilities for $V(f)$, e.g. $V(x-y^i)$ where $x,y$ is
a
part of a regular system of parameters.

Inductively we get that there is $P  \in U$ s.t. $A/P$ is regular
$1$-dimensional.

(3) If $A$ is a regular local ring and $P,P'$  different prime ideals with
$A/P$
and $A/P'$ regular one dimensional, then there is a prime ideal $Q \subset P
\cap P'$ with $A/Q$ regular $2$-dimensional.

Indeed let $x_1,\dots,x_n$ be a minimal system of generators of $P$; their
images
in
$A/P'$ generate a principal ideal; we may assume this ideal is generated by
the
image of $x_1$, and then we can substract some multiples of $x_1$ from
$x_2,\dots,x_n$ so that the images of $x_2,\dots,x_n$ are $0$; take
$Q=(x_2,\dots,x_n)$.

To prove the theorem we may assume $F$ is a finitely generated extension of
$k$, so $F$ is a finite extension of a purely transcendental extension $k'$
of
$k$. We replace $A$ by the local ring at the generic point of the special
fiber of
an affine space over $A$ that has residue field $k'$. So we reduce to $F/k$
finite.
Let $x,y$ be $F$-points of $X'_s$ centered at closed points $a,b$ at which
$X'$ is
regular. Let $U$ be dense open of $X$ above which $X'\to X$ is an
isomorphism. Let
$X'(a),X'(b)$ be the local schemes (Spec of the local rings at $a$ and $b$).
There
are regular one dimensional closed subschemes
$$C \subset X'(a), C' \subset X'(b)$$
whose generic points map to $U$.

By EGA $0_{III}$ 10.3 there are finite flat $D\to C$, $D'\to C'$ which are
$Spec(F)$ over the closed points of $C,C'$. Then $D,D'$ are Spec's of DVRs
essentially of finite type over $A$ (localization of finite type
$A$-algebras). We
form the pushout of $D\leftarrow Spec(F)\to D'$, which is Spec of a fibered
product ring, which by some algebraic exercise is still an $A$-algebra
essentially
of finite type. The pushout can be embedded as a closed subscheme in Spec of
a
local ring of an affine space over $A$ and then by (3) in  some $Y$ a
$2$-dimensional local regular $A$-scheme essentially of finite type. Now
$D,D'$
are subschemes of $Y$. We have a rational map $Y\to X'$ defined on the
inverse
image of $U$ and in particular at the generic points of $D$ and $D'$. By
e.g. Theorem 26.1 in Lipman's paper on rational singularities (Publ. IHES 36)
there is
$Y'\to Y$ obtained as a succession of blow-ups at closed points s.t. the
rational
map gives a morphism $Y'\to X'$. Then $x,y$ are images of $F$-points of $Y'$
(closed points of the proper transforms of $D,D'$),and by (1) any two
$F$-points
of the special fiber of $Y'\to Y$ are
$R$-equivalent.
\medskip

\enlargethispage*{20pt}

Sincerely,

\hfill Ofer Gabber

\end{document}